\documentclass[11pt]{article}
\usepackage[utf8]{inputenc}
\usepackage[english]{babel}
\usepackage{authblk}
\usepackage{graphicx}        
\usepackage{multicol}        
\usepackage[bottom]{footmisc}
\usepackage{placeins}
\usepackage{changes,todonotes}
\usepackage{pgf}
\usepackage{psfrag}
\usepackage{pgfplots}
\usepackage{hyperref}
\usepackage{appendix}
\usepackage{amsmath,amssymb}
\usepackage{bm}
\usepackage{mathtools}
\usepackage{color}
\usepackage{stmaryrd}
\usepackage{amsthm}
\usepackage{blindtext}
\usepackage{caption}
\usepackage{multirow}

\graphicspath{./FIGURES/}

\usepackage[margin=3cm]{geometry}

\usepackage[ruled,vlined]{algorithm2e}

\graphicspath{./FIGURES/}
\usepackage{subcaption}

\newcommand{\vertiii}[1]{{\left\vert\kern-0.25ex\left\vert\kern-0.25ex\left\vert #1 
    \right\vert\kern-0.25ex\right\vert\kern-0.25ex\right\vert}}

\newcommand{\trinorm}[1]{{\vert\kern-0.25ex\vert\kern-0.25ex\vert #1 \vert\kern-0.25ex\vert\kern-0.25ex\vert}}

\definecolor{farbe}{gray}{0.80}

\newtheorem{theorem}{Theorem}
\newtheorem{proposition}{Proposition}
\newtheorem{definition}{Definition}
\newtheorem{lemma}{Lemma}

\newtheorem{assumption}{Assumption}
\newtheorem{remark}{Remark}

\begin{document}

\title{A polytopal discontinuous Galerkin method for the pseudo-stress formulation of the unsteady Stokes problem}

\author[$\star$]{Paola F. Antonietti}
\author[$\star$]{Michele Botti}
\author[$\star$]{Alessandra Cancrini}
\author[$\star$]{Ilario Mazzieri}

\affil[$\star$]{MOX, Laboratory for Modeling and Scientific Computing, Dipartimento di Matematica, Politecnico di Milano, Piazza Leonardo da Vinci 32, I-20133 Milano, Italy}

\affil[ ]{\texttt {\{paola.antonietti,alessandra.cancrini\\ michele.botti,ilario.mazzieri\}@polimi.it}}

\maketitle

	
\begin{abstract}
This work aims to construct and analyze a discontinuous Galerkin method on polytopal grids (PolydG) to solve the pseudo-stress formulation of the unsteady Stokes problem. The pseudo-stress variable is introduced due to the growing interest in non-Newtonian flows and coupled interface problems, where stress assumes a fundamental role.
The space-time discretization of the problem is achieved by combining the PolydG approach with the implicit $\theta$-method time integration scheme. For both the semi- and fully-discrete problems we present a detailed stability analysis. Moreover, we derive convergence estimates for the fully discrete space-time discretization. A set of verification tests is presented to verify the theoretical estimates and the application of the method to cases of engineering interest. 
\end{abstract}

\section{Introduction}\label{intro}
The interaction between a fluid and a poroelastic structure is a complex problem that couples the Stokes equations with the poroelasticity system describing the motion of flows in deformable saturated porous mediums (see, e.g. \cite{Cesmelioglu2017,Showalter2005}).
Numerical modeling of this problem finds important applications, for example, in the simulation of the groundwater flow in fractured aquifers, oil and gas extraction, and biological flows.
In this work, we will focus on the Stokes problem, which models incompressible viscous free flows, formulating it in the pseudo-stress unknown rather than in its classical expression. 
Due to the recent growing interest in non-Newtonian fluid flow models (see e.g. \cite{Baaijens1998,Keunings2001}), which are crucial for the understanding of real fluids and have significant applications in the biological, medical, and industrial fields, the stress-velocity-pressure formulation for incompressible flows has garnered attention (see e.g. \cite{Gerritsma1999, Howell2009, Qiu.Zhao2024}). 
Indeed, for complex nonlinear flow problems, the use of a formulation where stress serves as a primal unknown can facilitate the design of approximation methods and its numerical solution \cite{Manouzi2001}.
In addition, an accurate approximation of the stress is crucial for determining traction on a fluid-solid interface.
Although stress can be a posteriori reconstructed in the velocity-pressure formulation through velocity differentiation, this compromises the accuracy. Note that a drawback associated with employing the stress-velocity-pressure formulation is the additional challenge introduced by the symmetry constraint of the stress tensor during the discretization process \cite{Arnold1984,Arnold2002}. 
One possible approach to overcome such a difficulty is based on employing the concept of pseudo-stress \cite{CaiLee2004}. The pseudostress-velocity formulation for the time-dependent Stokes problem has been suggested in \cite{Cai2007}, where only the implicit time discretization is mentioned, without providing the weak formulation and the theoretical analysis of the problem.

The pseudo-stress variable is defined as $\boldsymbol{\sigma}(\boldsymbol{u},p) = \mu\boldsymbol{\nabla}\boldsymbol{u} - p\mathbb{I}_d$, where $\boldsymbol{u}$ is the flow velocity and $p$ its pressure. The pseudo-stress, being nonsymmetric, allows for the adaptation of stable pairs designed for Darcy flows to the pseudostress-velocity formulation of the Stokes' equations \cite{Cai2005}. 
For this reason, in this paper, we formulate the unsteady Stokes problem as a single equation in the pseudo-stress variable.
For the resulting problem, we propose and analyze a discontinuous
Galerkin method on polytopal grids (PolydG). For previous results in the field of discontinuous Galerkin methods on polytopal grids, see e.g. \cite{Bassi2012,Antonietti2013,Cangiani2014,Cangiani2022,Cangiani2016,Congreve2019} addressing second-order elliptic problems, \cite{Cangiani2017} focusing on parabolic differential equations, \cite{Antonietti2019} modeling flows in fractured porous media, and \cite{AntoniettiVerani2019} addressing fluid-structure interaction problems. Additionally, we refer to \cite{CangianiDongGeorgoulisHouston_2017} for a comprehensive monograph. 
More recent dG discretizations on polytopal meshes can be found in \cite{AntoniettiMazzieri2018} for elastodynamics problems, in \cite{AntoniettiMazzieri2020} for nonlinear sound waves, in \cite{Antonietti2020_CMAME,AntoniettiBonaldi2020} for coupled elasto-acoustic problems, in \cite{Antonietti2023_SISC,bonetti2023numerical} for  thermo-elasticity, 
in \cite{Antonietti2021,Antonietti2022_VJM}  for poroelasto-acustics, and in  \cite{Corti2023_M3AS,Corti2023,fumagalli2023polytopal} for multiphysics brain modeling. Moreover, dG methods for a pure-stress formulation of the elasticity eigenproblem are proposed in \cite{Meddahi2023,Meddahi2019,Hong2021}, while in \cite{Lepe2020} dG methods for a pseudo-stress formulation of the Stokes eigenvalue problem are presented.

This work introduces a new contribution by presenting a comprehensive analysis of the proposed PolydG approximation of the Stokes problem in its pseudo-stress formulation. We first provide rigorous proof of the well-posedness and stability of the pseudo-stress weak formulation of the continuous problem, which is both novel and original. 
We then design and analyze both the semi-discrete and fully-discrete formulations based on the PolydG spatial discretization and the $\theta$-method time integration, carrying out a detailed stability analysis for both, and obtaining \textit{a priori} estimates. Finally, we present a convergence analysis for the fully-discrete problem, establishing error estimates in a suitable discrete norm.

The structure of the paper is as follows. In Section~\ref{sec:stokes_problem} we present the unsteady Stokes model problem in the pseudo-stress formulation, discuss the analogy with its classical velocity-pressure formulation, and prove the well-posedness of the weak formulation.  
In Section~\ref{sec:dG_discretization}, we introduce the semi-discrete setting and present the PolydG discretization of our problem. Next, we prove the stability of the discrete solution in the semi-discrete and fully-discrete setting, the latter obtained by coupling the Crank-Nicolson time integration scheme with the PolydG space discretization. Finally, in Section~\ref{sec:convergence} we prove a convergence result for the fully discrete problem. 
\textcolor{black}{Numerical tests that confirm the theoretical estimates are reported in Section~\ref{sec:numerical_test} together with an application of engineering interest, namely the flow around a cylinder, and two 3D test cases.} Finally, in Section~\ref{sec:conclusions} we draw some conclusions and an outlook of possible extension of this work.

\section{The model problem}\label{sec:stokes_problem}

\subsection{Notation}
Let $\Omega\subset\mathbb{R}^d$, $d=2,3$, be an open, convex polygonal/polyhedral domain with Lipschitz boundary $\partial\Omega$.
In what follows, for $X \subseteq \overline{\Omega}$, the notation $\bm{L}^2(X)$ is adopted in place of $[L^2(X)]^d$, while $\mathbb{L}^2(X)$ in place of $[L^2(X)]^{d\times d}$, with $d=2,3$. The scalar product in $L^2(X)$ is denoted by $(\cdot,\cdot)_X$, with the associated norm $\| \cdot \|_X$.  
Similarly, the Sobolev spaces $\bm{H}^\ell(X)$ are defined as $[H^\ell(X)]^d$, with $\ell\geq 0$, equipped with the norm $\| \cdot \|_{\ell,X}$, assuming conventionally that $\bm{H}^0(X)\equiv\bm{L}^2(X)$. 
In addition, we will use $\bm{H}(\textrm{div},X)$ to denote the space of $\bm{L}^2(X)$ vector fields with square-integrable divergence. Then, the notation $\mathbb{H}(\textrm{div},X)$ is used for the space of tensor fields with rows belonging to $\bm{H}(\textrm{div},X)$, equipped with the norm 
\begin{align*}
\left\|\boldsymbol{\sigma}\right\|^2_{\textrm{div}, X} = \left\|\boldsymbol{\sigma}\right\|^2_{\mathbb{L}^{2}(X)} + \left\|\boldsymbol{\nabla}\cdot\boldsymbol{\sigma}\right\|^2_{\bm L^{2}(X)} \ \ \forall \bm \sigma \in \mathbb{H}({\rm div}, X).
\end{align*} 
Finally, \textcolor{black}{for $\Gamma\subseteq\partial\Omega$ closed subset of $\partial\Omega$}, we consider the space $H^{\frac12}(\Gamma) = \{ v \in L^2(\Gamma)\, |\, \exists u\in H^1(\Omega): u_{|\Gamma} = v \}$ and its dual space $H^{-\frac{1}{2}}(\Gamma)$. 
The duality product between $v\in H^{\frac12}(\Gamma)$ and $w\in H^{-\frac12}(\Gamma)$ is denoted by $\langle v,w \rangle_{\Gamma}$.
For a given final time $T>0$, $k\in\mathbb{N}$, and a Hilbert space $H$, the usual notation $C^k([0,T];H)$ is adopted for the space of $H$-valued functions, $k$-times continuously differentiable in $[0,T]$. For
the sake of readability, we will use the notation $\boldsymbol{\dot \sigma}$ to denote the time derivative
of $\boldsymbol{ \sigma}$.
The notation $x\lesssim y$ stands for $x\leq C y$, with $C>0$, independent of the discretization parameters, but possibly dependent on the physical coefficients, the final time $T$, the domain $\Omega$, and the dimension $d$.

\subsection{Unsteady Stokes problem}\label{sec-sub:stokes_pb}
We start from the unsteady Stokes problem written in the classical velocity-pressure formulation: Find $(\boldsymbol{u}, p)$ such that

\begin{equation}\label{eq:stokes_primal}
\begin{cases}
  \dfrac{\partial \boldsymbol{u}}{\partial t} 
  -\mu\boldsymbol{\Delta} \boldsymbol{u} + \boldsymbol{\nabla} p = \boldsymbol{f}, 
  \qquad&\text{in }\Omega\times(0,T], \\
  \boldsymbol{\nabla} \cdot \boldsymbol{u} = 0, 
  \qquad &\text{in }\Omega\times(0,T], \\
  \boldsymbol{u} = \boldsymbol{u}_D, \;\; &{\rm on} \ \Gamma_D\times(0,T], \\
  (\mu\boldsymbol{\nabla} \boldsymbol{u} - p \mathbb{I}_d)\boldsymbol{n} = \boldsymbol{g}_N,\;\; 
  &{\rm on}\ \Gamma_N\times(0,T], \\
  \boldsymbol{u}(\cdot,t=0) = \boldsymbol{u}_0, &\text{in }\Omega, 
\end{cases}    
\end{equation}
\\
where $\boldsymbol{u}$ is the flow velocity, $p$ is the fluid pressure, $\mu>0$ is the fluid viscosity and $T>0$ is the final simulation time. The boundary of $\Omega$ is partitioned as $ \Gamma_D\cup \Gamma_N = \partial\Omega$, with $\Gamma_D\cap \Gamma_N = \emptyset$. For simplicity, we assume both $\left|  \Gamma_D \right| > 0$ and $\left|   \Gamma_N \right| > 0$, with $|\cdot|$ denoting the Hausdorff measure.

Setting $\bm H^1_{0,\Gamma_D}(\Omega)= \{\boldsymbol{v} \in \bm H^1(\Omega)^d \;|\; \boldsymbol{v}=\boldsymbol{0} \; \text{on}\; \Gamma_D\}$ and assuming for simplicity that $\boldsymbol{u}_D=\boldsymbol{0}$ and the forcing term $\boldsymbol{f}$ and traction $\boldsymbol{g}_N$ are regular, i.e. $\boldsymbol{f}\in \boldsymbol{L}^2(\Omega)$ and $\boldsymbol{g}_N\in \boldsymbol{L}^2(\Gamma_N)$, the weak formulation of \eqref{eq:stokes_primal} reads as: 
for any time $t \in (0, T]$ find $(\bm u, p)(t) \in \bm H^1_{0,\Gamma_D}(\Omega) \times  L^2(\Omega)$ such that 
\begin{align*}
&\left(\frac{\partial\boldsymbol{u}}{\partial t}, \boldsymbol{v}\right)_\Omega +
( \mu \boldsymbol{\nabla} \boldsymbol{u}, \boldsymbol{\nabla} \boldsymbol{v})_\Omega
- (\boldsymbol{\nabla} \cdot \boldsymbol{v},  p)_\Omega
= ( \boldsymbol{f}, \boldsymbol{v})_\Omega +  (\boldsymbol{g}_N , \boldsymbol{v})_{\Gamma_N}, \\ &
 \ (\boldsymbol{\nabla} \cdot \boldsymbol{u}, \ q )_\Omega  = 0,
\end{align*}
 for any $(\bm v, q) \in \bm H^1_{0,\Gamma_D}(\Omega) \times  L^2(\Omega)$.



\subsection{Pseudo-stress weak formulation}

We rewrite the Stokes problem \eqref{eq:stokes_primal} in a different form by introducing the pseudo-stress $\boldsymbol{\sigma}(\boldsymbol{u},p) = \mu\boldsymbol{\nabla}\boldsymbol{u} - p\mathbb{I}_d$  variable, where $\mathbb{I}_d$ is the identity matrix in $\mathbb{R}^{d\times d}$. Then, \eqref{eq:stokes_primal} can be rewritten as

\begin{equation}\label{eq:stokes_dual}
\begin{cases}
  \dfrac{\partial \boldsymbol{u}}{\partial t} 
  -\boldsymbol{\nabla}\cdot\boldsymbol{\sigma} = \boldsymbol{f}, 
  \qquad&\text{in }\Omega\times(0,T], \\
  \mu^{-1}{\rm \textbf{dev}}(\boldsymbol{\sigma}) - \boldsymbol{\nabla}\boldsymbol{u} = \bm 0, 
  \qquad &\text{in }\Omega\times(0,T], \\
  \boldsymbol{u} = \boldsymbol{u}_D, \;\; &{\rm on} \ \Gamma_D\times(0,T], \\
  \boldsymbol{\sigma}\ \boldsymbol{n} = \boldsymbol{g}_N,\;\; 
  &{\rm on}\  \Gamma_N\times(0,T], \\
  \boldsymbol{u}(\cdot,t=0) = \boldsymbol{u}_0, &\text{in }\Omega, 
\end{cases}
\end{equation}
where the deviatoric operator is defined such that 
${\rm \textbf{dev}}(\boldsymbol{\tau}) = \boldsymbol{\tau} - \frac1d {\rm tr}(\boldsymbol{\tau})\mathbb{I}_d$, and ${\rm tr}(\cdot)$ is the trace operator.
Notice that the incompressibility constraint $\boldsymbol{\nabla}\cdot\boldsymbol{u} = 0$ is enforced through the second equation of the previous system that yields ${\rm tr}(\boldsymbol{\nabla}\boldsymbol{u}) = \boldsymbol{\nabla}\cdot\boldsymbol{u} = 0$. 

Assuming enough regularity of the problem's data and solution, we can differentiate in time the second and third equations in \eqref{eq:stokes_dual} and plug the expression of $\frac{\partial \boldsymbol{u}}{\partial t}$ that we get from the first equation into the second one. Thus, we obtain:
\begin{equation}
\begin{cases}
  \frac1\mu \frac{\partial {\rm \textbf{dev}}(\boldsymbol{\sigma})}{\partial t}
  - \boldsymbol{\nabla}\left(\boldsymbol{\nabla}\cdot\boldsymbol{\sigma}\right)=\boldsymbol{F}, 
  \qquad &\text{in }\Omega\times(0,T], \\
  \boldsymbol{\nabla}\cdot\boldsymbol{\sigma}=\boldsymbol{g}_D, \;\; &{\rm on}\ \Gamma_D\times(0,T], \\
  \boldsymbol{\sigma}\ \boldsymbol{n} = \boldsymbol{g}_N,\;\; 
  &{\rm on}\ \Gamma_N\times(0,T], \\
  {\rm \textbf{dev}}(\boldsymbol{\sigma})(\cdot,t=0) = \boldsymbol{\sigma}_0, &\text{in }\Omega, 
\end{cases}
\label{eq:eqnn}
\end{equation}
with $\boldsymbol{F}= \boldsymbol{\nabla}\boldsymbol{f}$ and $\boldsymbol{g}_D = \frac{\partial\boldsymbol{u}_D}{\partial t}-\boldsymbol{f}_{|\Gamma_D}$. In this way, we reformulate problem \eqref{eq:stokes_primal} only in the pseudo-stress variable $\boldsymbol{\sigma}$. \textcolor{black}{We assume $\boldsymbol{g}_N$ and  $\boldsymbol{u}_D$ regular, i.e. $\boldsymbol{g}_N \in \boldsymbol{L}^2(\Gamma_N)$ and $\boldsymbol{u}_D \in H^1((0,T];\boldsymbol{H}^\frac{1}{2}(\Gamma_D))$}. We remark that we have also replaced the initial condition on $\boldsymbol{u}(\cdot,t=0)$ with a condition on ${\rm \textbf{dev}}(\boldsymbol{\sigma})(\cdot,t=0)$, \textcolor{black}{with initial datum $\boldsymbol{\sigma}_0 \in \mathbb{L}^2(\Omega)$}. This can be done under the assumption that the velocity solution is sufficiently regular.

In order to strongly enforce the essential traction condition on $\Gamma_N$, we define the subspace 
$$
\mathbb{H}_{0,\Gamma_N}({\rm div}, \Omega) = \{\boldsymbol{\eta}\in\mathbb{H}({\rm div}, \Omega)\ | \ \langle\bm\eta\ \bm n,\bm v\rangle_{\partial\Omega}=0 \ \ \forall \bm v\in \bm H^1_{0,\Gamma_D}(\Omega)\}.
$$
Then, assuming for simplicity $\boldsymbol{g}_N = \boldsymbol{0}$, testing the first equation with $\boldsymbol{\tau}\in \mathbb{H}_{0,\Gamma_N}({\rm div}, \Omega)$ and integrating by parts, we obtain the following weak formulation: for any $t\in (0, T]$, find $\bm \sigma(t) \in \mathbb{H}_{0,\Gamma_N}({\rm div}, \Omega)$ such that  
\begin{equation}\label{eq:weak_stress}
     (\mu^{-1} \partial_t{\rm \textbf{dev}}(\boldsymbol{\sigma}), {\rm \textbf{dev}}(\boldsymbol{\tau}))_{\Omega} + (\boldsymbol{\nabla}\cdot\boldsymbol{\sigma}, \boldsymbol{\nabla}\cdot\boldsymbol{\tau})_{\Omega} =  (\boldsymbol{F}, \boldsymbol{\tau})_{\Omega}  +  \langle\bm g_D, \boldsymbol{\tau}\,\boldsymbol{n}\rangle\textcolor{black}{_{\partial \Omega}}
\end{equation}
for any $\boldsymbol{\tau}\in \mathbb{H}_{0,\Gamma_N}({\rm div}, \Omega)$.
We point out that, owing to the definition of the stress space $\mathbb{H}_{0,\Gamma_N}({\rm div}, \Omega)$, the duality pairing on the right-hand side of (\ref{eq:weak_stress}) is taken over the whole boundary $\partial \Omega$, and not only over $\Gamma_D$. 
\textcolor{black}{
Therefore, we make the following assumption.
\begin{assumption}
 We assume that, for any time $t$, the Dirichlet datum $\boldsymbol{g}_D=\boldsymbol{g}_D(t)$ is the trace of a function in $\boldsymbol{H}^1(\Omega)$, which, with a slight abuse of notation we still denote by $\boldsymbol{g}_D$, so that $\boldsymbol{g}_D(t)\in  \bm H^{\frac12}(\partial \Omega)$. This allows us to write the duality over $\partial\Omega$ in \eqref{eq:weak_stress}.
\end{assumption}
}
\textcolor{black}{
\begin{remark}[Data regularity]
For the terms appearing in the right-hand side of \eqref{eq:weak_stress} to be well-defined, for all $t\in (0, T]$ we would require $\boldsymbol{F}(t)\in \mathbb{L}^2(\Omega)$ and $\boldsymbol{g}_D (t)\in \bm H^{\frac12}(\textcolor{black}{\partial \Omega})$, which imply $\bm f(t)\in \bm H^1(\Omega)$. This assumption can be weakened by using integration by parts to rewrite the right-hand side of \eqref{eq:weak_stress} as 
$$
(\boldsymbol{F}, \boldsymbol{\tau})_{\Omega}  
+  \langle\bm g_D, \boldsymbol{\tau}\,\boldsymbol{n} \rangle\textcolor{black}{_{\partial \Omega}}=
-(\boldsymbol{f}, \boldsymbol{\nabla}\cdot\boldsymbol{\tau})_{\Omega} + \langle \partial_t \bm u_D, \boldsymbol{\tau}\,\boldsymbol{n} \rangle\textcolor{black}{_{\partial \Omega}}.
$$
\end{remark}
}
\begin{remark}[Pressure and velocity recovery]\label{rem-sub:p_and_u_rec}
It is possible to compute the pressure and velocity fields from the solution of \eqref{eq:weak_stress}.
The pressure can be easily obtained from the relation $p = -\frac1d {\rm tr}(\boldsymbol{\sigma})$, while to recover the velocity $\bm u$, we use the fundamental theorem of calculus
and the first equation in \eqref{eq:stokes_dual} to  obtain
\begin{equation*}\label{eq:recovery_vel}
 \boldsymbol{u}(t) = \boldsymbol{u}_0 + \int_{0}^{t} \Big( \boldsymbol{\nabla}\cdot\boldsymbol{\sigma}(s) + \boldsymbol{f}(s)  \Big) \ \text{d}s
\end{equation*}
for $t \in (0,T]$.
\end{remark}
%

\subsection{Well-posedness of the continuous problem}

The existence and uniqueness of the solution to problem \eqref{eq:weak_stress} can be inferred in the framework of degenerate implicit evolution equations. In particular, we hinge on the following result; its proof along with examples and further details, is provided in \cite[Chapter V, Theorem 4.1]{Showalter1994}. 
\begin{proposition}
\label{th:abstract_well_posedness}
Let $V_m$ be a seminorm space obtained from a symmetric and non-negative bilinear form $m(\cdot,\cdot)$ and let $\mathcal{M}:V_m\to V_m'$ be the linear operator defined by $\mathcal{M}x(y) = m(x,y)$ for any $x,y \in V_m$. Let $D$ be a dense subspace of $\ V_m$ and $\mathcal{L}: D \rightarrow V_m'$ be linear and monotone. Let $K(\mathcal{M})$ and $K(\mathcal{L})$ denote the kernels of the operators $\mathcal{M}$ and $\mathcal{L}$, respectively. It holds:
\begin{enumerate}
    \item[(a)] if $\ K(\mathcal{M}) \cap D \subseteq K(\mathcal{L})$ and if $\mathcal{M} + \mathcal{L}: D \rightarrow V_m'$ is onto, then for every $f \ \in \ C^1([0,T],V_m')$ and $u_0 \in D$ there exists a solution of $\ 
    \partial_t(\mathcal{M}u)(t)  + \mathcal{L}u(t)  =  f(t), \ t>0,$ with $
    (\mathcal{M}u)(0) = \mathcal{M}u_0$;
    \item[(b)] if $\ K(\mathcal{M}) \cap K(\mathcal{L}) = \textcolor{black}{\{ \Theta \}}$, then there is at most one solution.
\end{enumerate} 
\end{proposition}

We rewrite the initial boundary value problem \eqref{eq:weak_stress} in the notation of the previous proposition assuming for simplicity $\Gamma_N=\partial\Omega$. Therefore, we define the operators $
    \mathcal{M} = \mu^{-1} {\rm \textbf{dev}}(\cdot): V_m\to V_m' \ \text{and} \ \mathcal{L} = - \boldsymbol{\nabla}\left(\boldsymbol{\nabla}\cdot(\cdot)\right): D\to V_m',
$
with
$$
V_m = \mathbb{L}^2(\Omega), \quad V_m' = \{\boldsymbol{\tau}\in \mathbb{L}^2(\Omega)\ | \ \textrm{tr}(\boldsymbol{\tau})=0\},$$ 
$$D = \{\boldsymbol{\tau}\in \mathbb{H}_{0,\Gamma_N}({\rm div}, \Omega)\ | \ \boldsymbol{\nabla}\cdot\boldsymbol{\tau}\in \boldsymbol{H}^1(\Omega)\ \text{ and }\ 
\boldsymbol{\nabla}\cdot\left(\boldsymbol{\nabla}\cdot\boldsymbol{\tau}\right)=0\}.
$$
We define $V_m'$ as described above since the definition of the deviatoric operator implies $\textrm{tr}(\textbf{dev}(\boldsymbol{\sigma})) = 0$. While $D$ is defined in this way because $\mathcal{L}: D\to V_m' \leftrightarrow \textrm{tr}(\boldsymbol{\nabla}\left(\boldsymbol{\nabla}\cdot \boldsymbol{\sigma}\right)) = \boldsymbol{\nabla}\cdot\left(\boldsymbol{\nabla}\cdot\boldsymbol{\sigma}\right)=0$.
Then, $D$ is dense in $V_m$ and $\mathcal{L}$ is continuous and strongly monotone.
We remark that $K(\mathcal{M}) \cap D = \{\psi\mathbb{I}_d \in V_m \ | \ \boldsymbol{\Delta} \psi=0\ \text{and} \ \psi = 0  \ \text{on} \ \Gamma_N\} = \textcolor{black}{\{ \Theta \}}$. As a result, it holds $K(\mathcal{M}) \cap D \subseteq K(\mathcal{L})$. It also holds that $K(\mathcal{M}) \cap K(\mathcal{L}) = \textcolor{black}{\{ \Theta \}}$. Therefore, the hypotheses of Proposition \ref{th:abstract_well_posedness} are satisfied and, as a result, the existence and uniqueness of the solution are guaranteed.

\subsection{Stability analysis of the continuous problem}

In this section, before proving the stability estimate for the solution of the weak problem \eqref{eq:weak_stress}, we introduce the following preliminary results. For the sake of presentation, we present the proof in Appendix~\ref{sec:appendix}. 

\begin{lemma}\label{lemma:dev_div} (${\rm dev-div}$ and trace inequalities). 
Let $\boldsymbol{\sigma} \in \mathbb{H}_{0,\Gamma_N}({\rm div}, \Omega)$, with $|\Gamma_N| > 0$. 
Then, there exists two positive constants $C_{\rm dd}$ and $C_{\rm tr}$ such that
\begin{align}
\label{eq:devdiv}
&\left\|\boldsymbol{\sigma}\right\|_{\mathbb{L}^{2}(\Omega)}\ \le C_{\rm dd} \left( \left\|
 \boldsymbol{{\rm dev}}(\boldsymbol{\sigma})\right\|_{\mathbb{L}^{2}(\Omega)}+\left\|\boldsymbol{\nabla}\cdot\boldsymbol{\sigma}\right\|_{\bm L^{2}(\Omega)}\right), \\ &
 \label{eq:traceHdiv}
\textcolor{black}{\left\|\boldsymbol{\sigma}\boldsymbol{n}\right\|_{\bm H^{-1/2}(\partial\Omega)}} \le C_{\rm tr} \left( \left\|
 \boldsymbol{{\rm dev}}(\boldsymbol{\sigma})\right\|_{\mathbb{L}^{2}(\Omega)}+\left\|\boldsymbol{\nabla}\cdot\boldsymbol{\sigma}\right\|_{\bm L^{2}(\Omega)}\right).
\end{align}
\end{lemma}

As a result of the previous lemma, we get that $\|\mu^{-\frac12}{\rm \textbf{dev}}(\boldsymbol{\sigma})\|_{\mathbb{L}^{2}(\Omega)} +\left\|\boldsymbol{\nabla}\cdot\boldsymbol{\sigma}\right\|_{\bm L^{2}(\Omega)} $ defines a (weighted) norm on the space $\mathbb{H}_{0,\Gamma_N}({\rm div}, \Omega)$. 
Hinging on the previous results, we establish a stability estimate for the pseudo-stress solution. 

\begin{theorem}[Stability estimate]\label{thm:stability_continuous}
Let $\bm \sigma(t) \in \mathbb{H}_{0,\Gamma_N}({\rm div}, \Omega)$  be the solution of \eqref{eq:weak_stress} for any time $t \in (0,T]$. Then, it holds
\begin{align*}
& \left\|\mu^{-\frac12}\boldsymbol{{\rm dev}}(\boldsymbol{\sigma})\right\|_{L^{\infty}(0,T;\mathbb{L}^{2}{(\Omega)})}^{2} + 
\int_{0}^{T} \left\|\boldsymbol{\nabla}\cdot\boldsymbol{\sigma}(t)\right\|_{\bm L^{2}(\Omega)}^{2} \, {\rm d} t \\
&\quad \lesssim (1+2\mu T) \int_{0}^{T} \Big( \left\|\boldsymbol{F}(t)\right\|_{\mathbb{L}^{2}(\Omega)}^{2} + \left\|\boldsymbol{g_D}(t)\right\|_{\textcolor{black}{\bm H^{1/2}(\partial \Omega)}}^{2} \Big) \ {\rm d} t + 
 \left\|\mu^{-\frac12}\boldsymbol{{\rm dev}}(\boldsymbol{\sigma}_0)\right\|_{\mathbb{L}^{2}(\Omega)}^{2},
\end{align*}
with hidden constant independent of the viscosity $\mu$ and the final time $T$.
\end{theorem}

\begin{proof}  
Taking $\boldsymbol{\tau} = \boldsymbol{\sigma}(t)$ in \eqref{eq:weak_stress}, yields
\begin{align*}
 (\mu^{-1} \partial_t {\rm \textbf{dev}}(\boldsymbol{\sigma}), {\rm \textbf{dev}}(\boldsymbol{\sigma}))_{\Omega}
& + (\boldsymbol{\nabla}\cdot\boldsymbol{\sigma}, \boldsymbol{\nabla}\cdot\boldsymbol{\sigma})_{\Omega}
  = (\boldsymbol{F}, \boldsymbol{\sigma})_{\Omega} +  \langle\bm g_D, \boldsymbol{\sigma}\boldsymbol{n} \rangle\textcolor{black}{_{\partial \Omega}}.
\end{align*}
The previous identity can be rewritten as
\begin{align*}
\frac{\partial}{\partial t} \left\|\mu^{-\frac12} {\rm \textbf{dev}}(\boldsymbol{\sigma})\right\|_{\mathbb{L}^{2}(\Omega)}^{2} + 2 \left\|\boldsymbol{\nabla}\cdot\boldsymbol{\sigma}\right\|_{\bm L^{2}(\Omega)}^{2} = 2(\boldsymbol{F}, \boldsymbol{\sigma})_{\Omega} + 2 \langle\bm g_D , \boldsymbol{\sigma}\boldsymbol{n}\rangle\textcolor{black}{_{\partial \Omega}}.
\end{align*}
Now, on the right-hand side, we use the Cauchy–Schwarz inequality, Lemma~\ref{lemma:dev_div}, and Young's inequality to get
\begin{align*}
|(\boldsymbol{F}, &\boldsymbol{\sigma})_{\Omega}| +  |\langle\bm g_D , \boldsymbol{\sigma}\boldsymbol{n}\rangle\textcolor{black}{_{\partial \Omega}}|  
\le \big(C_{\rm dd}\left\|\boldsymbol{F}\right\|_{\mathbb{L}^{2}(\Omega)} +  C_{\rm tr}\left\|\boldsymbol{g_D}\right\|_{\textcolor{black}{\bm H^{1/2}(\partial \Omega)}}\big) \left( \left\|
 \boldsymbol{{\rm dev}}(\boldsymbol{\sigma})\right\|_{\mathbb{L}^{2}(\Omega)}+\left\|\boldsymbol{\nabla}\cdot\boldsymbol{\sigma}\right\|_{\bm L^{2}(\Omega)}\right)
\\ &\le 
\frac{\mathcal{R}^2}{2} + 
\frac{\left\|\mu^{-\frac12}\boldsymbol{{\rm dev}}(\boldsymbol{\sigma})\right\|_{\mathbb{L}^{2}(\Omega)}^{2}}{4T} + \frac{\left\|\boldsymbol{\nabla}\cdot\boldsymbol{\sigma}\right\|_{\bm L^{2}(\Omega)}^{2}}{2},
\end{align*}
where we have introduced $\mathcal{R} = \sqrt{1+2\mu T} \big(C_{\rm dd}\left\|\boldsymbol{F}\right\|_{\mathbb{L}^{2}(\Omega)} +  C_{\rm tr}\left\|\boldsymbol{g_D}\right\|_{\textcolor{black}{\bm H^{1/2}(\partial \Omega)}}\big)$.
Therefore, it is inferred that
$$
\frac{\partial}{\partial t} \left\|\mu^{-\frac12} {\rm \textbf{dev}}(\boldsymbol{\sigma})\right\|_{\mathbb{L}^{2}(\Omega)}^{2} +  \left\|\boldsymbol{\nabla}\cdot\boldsymbol{\sigma}\right\|_{\bm L^{2}(\Omega)}^{2}\le
\mathcal{R}^2+\frac1{2T}\left\|\mu^{-\frac12}\boldsymbol{{\rm dev}}(\boldsymbol{\sigma})\right\|_{\mathbb{L}{2}(\Omega)}^{2}.
$$
We now integrate the previous bound over the time interval $(0,t)$, obtaining
\begin{align}\label{inequal}
& \left\|\mu^{-\frac12} \boldsymbol{{\rm dev}}(\boldsymbol{\sigma})(t)\right\|_{\mathbb{L}^{2}(\Omega)}^{2} + \int_{0}^{t} \left\|\boldsymbol{\nabla}\cdot\boldsymbol{\sigma}(s)\right\|_{\bm L^{2}(\Omega)}^{2}\, {\rm d} s  \nonumber \\ &\;  \le
 \left\|\mu^{-\frac12}\boldsymbol{{\rm dev}}(\boldsymbol{\sigma}_0)\right\|_{\mathbb{L}^{2}(\Omega)}^{2} +  
 \int_{0}^{t} \Big( \mathcal{R}^2(s) + \frac1{2T}\left\|\mu^{-\frac12}\boldsymbol{{\rm dev}}(\boldsymbol{\sigma})(s)\right\|_{\mathbb{L}^{2}(\Omega)}^{2}  \Big) \ {\rm d} s.
\end{align}
Then, taking the supremum in $(0, T]$ on both sides of \eqref{inequal}, we get 
\begin{align*}
& \sup_{t\in(0,T]}\left\|\mu^{-\frac12}\boldsymbol{{\rm dev}}(\boldsymbol{\sigma})(t)\right\|_{\mathbb{L}^{2}{(\Omega)}}^{2}
+ \int_{0}^{T} \left\|\boldsymbol{\nabla}\cdot\boldsymbol{\sigma}(t)\right\|_{\bm L^{2}(\Omega)}^{2}\, {\rm d} t 
\\ &\quad \le
\left\|\mu^{-\frac12}\boldsymbol{{\rm dev}}(\boldsymbol{\sigma}_0)\right\|_{\mathbb{L}^{2}(\Omega)}^{2}  + \frac12 \sup_{t\in(0,T]}\left\|\mu^{-\frac12}\boldsymbol{{\rm dev}}(\boldsymbol{\sigma})(t)\right\|_{\mathbb{L}^{2}(\Omega)}^{2} +  
 \int_{0}^T \mathcal{R}^2(t) \, {\rm d} t .
\end{align*}
Rearranging the previous inequality by moving the second term on the right-hand side to the left-hand side, we get the conclusion.
\end{proof}

\section{Numerical discretization}\label{sec:dG_discretization}

\subsection{PolydG semi-discrete formulation} \label{polydg_semi}
In this section, we introduce the PolydG semi-discrete formulation of \eqref{eq:weak_stress}. Let $\mathcal{T}_h$ be a polytopal mesh of the domain $\Omega$, i.e.,   $\mathcal{T}_h = \textcolor{black}{\bigcup_{\kappa} \kappa}$, being  $\kappa$ a general polygon ($d = 2$) or polyhedron ($d = 3$). 
Given a polytopal element $\kappa$, we define by $\left|  \kappa \right|$ its measure and by \textcolor{black}{$h_\kappa$} its diameter, and set \textcolor{black}{$h = \max_{\kappa \in \mathcal{T}_h} h_\kappa$}. We let a polynomial degree \textcolor{black}{$p_\kappa \ge 1$} be associated with each element $\kappa \in \mathcal{T}_h$ and we denote by $p_h : \mathcal{T}_h \rightarrow \mathbb{N^*} = \{n \in \mathbb{N} : n \ge 1\}$ the piecewise constant function such that $(p_h)|_{\kappa} = \textcolor{black}{p_\kappa}$. Then, we define the discrete space
$$\boldsymbol{V}_h = [\textit{P}_{p_h} (\mathcal{T}_h)]^{d\times d}, \ \text{where} \ \textit{P}_{p_h} (\mathcal{T}_h) = \Pi_{\kappa \in \mathcal{T}_h} \mathbb{P}_{p_\kappa} (\kappa)$$ and $\mathbb{P}_{\ell} (\kappa)$ is the space of piecewise polynomials in $\kappa$ of total degree less than or equal to $\ell \geq 1.$ 
 We define an interface as the intersection of the ($d - 1$) - dimensional faces of any two neighboring
elements of $\mathcal{T}_h$. If $d = 2$, an
interface/face is a line segment and the set of all interfaces/faces is denoted by $\mathcal{F}_h$.
When $d = 3$, an interface can be a general polygon that we assume could be further decomposed into a set of planar triangles collected in the set $\mathcal{F}_h$.
We also decompose the set of faces as $\mathcal{F} = \mathcal{F}_h^I \cup \mathcal{F}_h^D  \cup \mathcal{F}_h^N$, where $\mathcal{F}_h^I$ contains the internal faces and $\mathcal{F}_h^D$ and $\mathcal{F}_h^N$ the faces of the Dirichlet and Neumann boundary, respectively. Following \cite{CangianiDongGeorgoulisHouston_2017}, we next introduce the main assumption on $\mathcal{T}_h$. 
\begin{definition}\label{def::polytopic_regular}
A mesh $\mathcal{T}_h$ is said to be \textit{polytopic-regular} if for any $ \kappa \in \mathcal{T}_h$, there exists a set of non-overlapping $d$-dimensional simplices contained in $\kappa$, denoted by $\{S_\kappa^F\}_{F\subset{\partial \kappa}}$, such that for any face $F\subset\partial \kappa$, it holds $h_\kappa\lesssim d|S_\kappa^F| \, |F|^{-1}$.
\end{definition}
\begin{assumption}
The sequence of meshes $\{\mathcal{T}_h\}_h$ is assumed to be \textit{uniformly} polytopic regular in the sense of  Definition ~\ref{def::polytopic_regular}. 
\label{ass::regular}
\end{assumption}
\noindent
As pointed out in \cite{CangianiDongGeorgoulisHouston_2017}, this assumption does not impose any restriction on either the number of faces per element or their measure relative to the diameter of the element they belong to.
In order to avoid technicalities, we also assume that $\mathcal{T}_h$ satisfies a \textit{hp-local bounded variation} property:
\begin{assumption}
For any pair of neighboring elements $\kappa^\pm\in\mathcal{T}_h$, it holds $h_{\kappa^+}\lesssim h_{\kappa^-}\lesssim h_{\kappa^+},\ \ p_{\kappa^+}\lesssim p_{\kappa^-}\lesssim p_{\kappa^+}$.
\label{ass::3}
\end{assumption}
\noindent Finally, for sufficiently piecewise smooth vector- and tensor-valued fields $\bm{v}$ and $\bm{\tau}$, respectively, and for any pair of neighboring elements $\kappa^+$ and $\kappa^-$ sharing a face $F\in \mathcal{F}_h^I$,  we introduce the jump and average operators
 $$
[[ \bm{v} ]]  = \bm{v}^+\otimes\bm{n}^++\bm{v}^-\otimes\bm{n}^-, \quad 
 [[\boldsymbol{\tau}]] = \boldsymbol{\tau}^+ \boldsymbol{n}^+ +  \boldsymbol{\tau}^- \mathbf{n}^-, 
 $$
  $$
 \{\!\{\bm{v}\}\!\} = \frac{ \bm{v}^+ +  \bm{v}^-}{2}, \quad 
 \{\!\{\boldsymbol{\tau}\}\!\} = \frac{\boldsymbol{\tau}^+  +  \boldsymbol{\tau}^-}{2}, \,\,
 $$
where $\otimes$ is the tensor product in $\mathbb{R}^d$, $\cdot^{\pm}$ denotes the trace on $F$ taken within $\kappa^\pm$, and $\bm{n}^\pm$ is the outer normal vector to $\partial \kappa^\pm$. 
Accordingly, on boundary faces  $F \in \mathcal{F}_h^D  \cup \mathcal{F}_h^N$, 
we set $[[\bm{v}]] = \bm{v}\otimes\bm{n}, \,\,
[[\boldsymbol{\tau}]] = \boldsymbol{\tau} \boldsymbol{n}, \,\, {\rm and }  \,\,
\{\!\{\bm{v}\}\!\} =  \bm{v}, \,\,
\{\!\{\bm{\tau}\}\!\} =  \bm{\tau}.$
In the following, we use $\nabla_h \cdot$ to denote the element-wise divergence operator, and we use the short-hand notation $(\cdot,\cdot)_{\mathcal{T}_h} = \sum_{\kappa\in \mathcal{T}_h}\int_{\kappa} \cdot$ and $ < \cdot,\cdot >_{\mathcal{F}_h} = \sum_{F\in \mathcal{F}_h}\int_{F} \cdot$.
We consider the following semi-discrete PolydG approximation to \eqref{eq:weak_stress}: for any $t\in (0,T]$, find $\bm \sigma_h(t) \in \bm V_h$ s.t.
\begin{equation}\label{eq:weak_dg}
   \begin{cases}
\mathcal{M}(\partial_t \bm \sigma_{h}, \bm \tau_h) + \mathcal{A} (\bm \sigma_h, \bm \tau_h) = F(\bm \tau_h) & \forall \, \bm \tau_h \in  \bm V_h,  \\
(\bm \sigma_h(0), \bm \tau_h) = (\bm \sigma_0, \bm \tau_h)  & \forall \, \bm \tau_h \in  \bm V_h,
   \end{cases}
\end{equation}
where for any $\bm \sigma, \bm \tau \in \bm V_h$ we have defined
\begin{eqnarray*}
    \mathcal{M}(\bm \sigma, \bm \tau) & = & ( \mu^{-1} \boldsymbol{{\rm dev}}(\boldsymbol{\sigma}), \boldsymbol{{\rm dev}}(\boldsymbol{\tau}) )_{\mathcal{T}_h},\\  
    \mathcal{A}(\bm \sigma, \bm \tau) & = & (\boldsymbol{\nabla}_h\cdot\boldsymbol{\sigma}, \boldsymbol{\nabla}_h\cdot\boldsymbol{\tau})_{\mathcal{T}_h} - \langle \{\!\{\boldsymbol{\nabla}_h\cdot\boldsymbol{\sigma}\}\!\} , [[\boldsymbol{\tau}\boldsymbol{n}]] \rangle_
    {\mathcal{F}_h^{I,N}} \\
    &  & - \langle \{\!\{\boldsymbol{\nabla}_h\cdot\boldsymbol{\tau}\}\!\} , [[\boldsymbol{\sigma}\boldsymbol{n}]] \rangle_
    {\mathcal{F}_h^{I,N}}    
    + \langle \gamma_e[[\boldsymbol{\sigma}\boldsymbol{n}]] , [[\boldsymbol{\tau}\boldsymbol{n}]] \rangle_
    {\mathcal{F}_h^{I,N}},  \\
    F(\bm \tau) & = &  (\boldsymbol{F}, \boldsymbol{\tau})_{\mathcal{T}_h} +
    \langle \boldsymbol{g}_D,  \boldsymbol{\tau}\boldsymbol{n} \rangle_{\mathcal{F}_h^D}
   +  \langle \boldsymbol{g}_N , \gamma_e \boldsymbol{\tau}\boldsymbol{n} + (\boldsymbol{\nabla}_h\cdot\boldsymbol{\tau})\rangle_{\mathcal{F}_h^N}.
\end{eqnarray*}
%
%
%
Here, $\mathcal{F}_h^{I,N} = \mathcal{F}_h^{I} \cup \mathcal{F}_h^{N}$ and the stabilization function $\gamma_e: \mathcal{F}_h^{I,N}\rightarrow\mathbb{R}_+$ is defined as
\begin{equation}\label{def:penalty}
    \gamma_e(\bm x) = \begin{cases}
        \alpha \max_{\kappa \in \{ \kappa^+,\kappa^-\}} \frac{p_\kappa^2}{h_\kappa}, & \bm x \in e, e \in \mathcal{F}^I_h, e \subset \partial \kappa^+ \cap \partial \kappa^-,  \\
        \alpha \frac{p_\kappa^2}{h_\kappa}, & \bm x \in e, e \in \mathcal{F}^N_h, e \subset \partial \kappa^+ \cap \partial \Gamma_N,
    \end{cases}
\end{equation}
being $\alpha>0$ the penalty coefficient at our disposal.
%
%

\subsection{Fully-discrete formulation}

We introduce a basis for the space $\bm V_h$ and express $\boldsymbol{\sigma}_h$ as a linear combination of these basis functions, where the unknown coefficients are functions of time.
We collect the latter in the vector $\underline{\bm \sigma}_h$, denote by $M$ (resp. $A$) the matrix representation of the bilinear form  $\mathcal{M}(\cdot,\cdot)$ (resp. $\mathcal{A}(\cdot, \cdot)$), and by $\underline{\bm f}$ the vector representation of the linear functional $F(\cdot)$. The algebraic formulation of \eqref{eq:weak_dg} reduces to: for any time $t\in(0,T]$, find $\underline{\bm \sigma}_h(t) \in \bm V_h$ s.t.
\begin{equation}\label{eq:ode_sigma}
    M\underline{\dot{\bm \sigma}}_h(t) + A\underline{\bm \sigma}_h(t) = \underline{\bm f}(t) \quad \forall t \in (0,T],
\end{equation}
with initial condition $\underline{\bm \sigma}_h(0) = \underline{\bm \sigma}_{0,h}$, being the latter the vector representation of the  $\bm L^2$-projection of $\bm \sigma_0$ onto  $\bm V_h$.  
In the following, to integrate in time \eqref{eq:ode_sigma} we use the $\theta$-method: for any $n=1,..., N_T$ find $\underline{\bm \sigma}_h^{n+1}$ such that
\begin{equation}\label{eq:tetametodo}
(M + \theta\Delta t A) \ \underline{\bm \sigma}_h^{n+1} = [M - (1 - \theta)\Delta t A] \ \underline{\bm \sigma}_h^{n} + \Delta t  [\theta\underline{\bm f}^{n+1} + (1 - \theta)\underline{\bm f}^{n}]
\end{equation}
with $N_T =\Delta t/T$, \textcolor{black}{$\theta \in [\frac{1}{2},1]$} and where the superscript $n$ means the approximation/evaluation of the given quantity at time $t_n = n \Delta t$, $n=0,..., N_T$.
\textcolor{black}{\begin{remark}
Note that, in this framework, an explicit time-stepping scheme cannot be employed since it would lead to an algebraic system with a singular matrix. See also Remark \ref{rem:tetametodo}.
\end{remark}}

\subsection{Stability analysis for the semi- and fully-discrete formulations} 
In the following, we present the stability analysis for both the semi- and fully-discrete problems and the \textit{a priori} error analysis. For the sake of presentation, we consider the case in which  $\boldsymbol{g}_N = \boldsymbol{g}_D = \bm 0$. 

\subsubsection{Stability analysis of the semi-discrete problem}

Before stating the main result of this section, we recall that
\begin{align*}
    |\boldsymbol{\sigma}|_{dG}^2 & = \left\|\nabla_h \cdot \boldsymbol{\sigma} \right\|^2_\Omega + \left\| \ \gamma^{1/2}[[\boldsymbol{\sigma}\boldsymbol{n}]] \ \right\|^2_{\mathcal{F}_h^{I,N}} & \forall \boldsymbol{\sigma} \in \boldsymbol{V}_h,
\end{align*}
\begin{align*}
   \trinorm{\boldsymbol{\sigma}}_{dG}^2 & = |\boldsymbol{\sigma}|_{dG}^2 + \left\| \ \gamma^{-1/2}\{\{\bm{\nabla}_h \cdot \boldsymbol{\sigma}\}\} \ \right\|^2_{\mathcal{F}_h} & \forall \boldsymbol{\sigma} \in \mathbb{H}^2(\mathcal{T}_h).
\end{align*}
\begin{theorem}
For any time $t \in (0,T]$, let $\bm \sigma_h(t) \in \bm V_h$  be the solution of \eqref{eq:weak_dg}. Then, for $\alpha$ in \eqref{def:penalty}  sufficiently large, it holds
\begin{align*}
& \left\|\mu^{-\frac12}\boldsymbol{{\rm dev}}(\boldsymbol{\sigma}_h)\right\|_{L^{\infty}(0,T;\mathbb{L}^{2}{(\Omega)})}^{2} + 
\int_{0}^{T} | \bm \sigma_h |_{dG}^2 \, {\rm d} t  \lesssim  \int_{0}^{T} \left\|\boldsymbol{F}\right\|_{\mathbb{L}^{2}(\Omega)}^{2} \, {\rm d} t + 
 \left\|\mu^{-\frac12}\boldsymbol{{\rm dev}}(\boldsymbol{\sigma}_{0,h})\right\|_{\mathbb{L}^{2}(\Omega)}^{2}.
\end{align*}
\end{theorem}
\begin{proof}
We choose $\boldsymbol{\tau}_h = \boldsymbol{\sigma}_h(t)$  in \eqref{eq:weak_dg} and apply  \eqref{eq:coercivity_dg}, Cauchy-Schwarz and Young inequalities to  obtain
\begin{align*}
    & \partial_t \left\|\mu^{-\frac12}\boldsymbol{{\rm dev}} (\boldsymbol{\sigma}_h)\right\|^2_{\mathbb{L}^2(\Omega)} +|\boldsymbol{\sigma}_h|_{dG}^2 \lesssim
 \frac{1}{\epsilon} \left\|\boldsymbol{F}\right\|^2_{\mathbb{L}^2(\Omega)} + \epsilon \left\|\boldsymbol{\sigma}_h\right\|^2_{\mathbb{L}^2(\Omega)}, 
\end{align*}
recalling that the coercivity bound \eqref{eq:coercivity_dg} holds provided that the stability parameter $\alpha$ is sufficiently large. 
Then, the proof is obtained by applying the same arguments as the ones in the proof of Theorem \ref{thm:stability_continuous} and using the results in Lemma~\ref{lemma:sigma_l2}. 
\end{proof}

The well-posedness of the semi-discrete problem \eqref{eq:weak_dg} can be inferred in the framework of differential-algebraic equations (DAEs). The DAEs theory states that, if the matrix pencil $sM+A$ is nonsingular for some $s\neq 0$, then problem \eqref{eq:ode_sigma} admits a solution \cite{Brenan1996}. Thus, to establish existence, it is enough to take $s=1$ and apply Lemma \ref{lemma:cont_coerc_A} and Lemma \ref{lemma:sigma_l2} to obtain that $M+A$ is positive definite. Finally, uniqueness and well-posedness follow from the linearity and the a priori stability estimate.

\subsubsection{Stability analysis of the fully-discrete problem}
In this section, we present the stability analysis for the fully discrete problem (\ref{eq:tetametodo}). For the sake of presentation, here and in the following, we consider the implicit Euler method (i.e., $\theta = 1$) to integrate in time the system \eqref{eq:weak_dg}. A similar proof can be obtained for the general case where $\theta \in [\frac12, 1)$, see Remark \ref{rem:tetametodo}.
We start by considering \eqref{eq:tetametodo} with $\theta = 1$: for any time $t_n$, for $n=1,..., N_T$ find $ \bm \sigma_h^n \in \bm V_h$ s.t.  
\begin{equation}\label{eq:semi_discrete_partial}
\begin{cases}
    \mathcal{M}(\bm{\sigma}_h^n -  \bm{\sigma}_h^{n-1}, \boldsymbol{\tau}_h ) + \Delta t  \mathcal{A}(\boldsymbol{\sigma}_h^n, \boldsymbol{\tau}_h) = \Delta t  ( \boldsymbol{F}^n, \boldsymbol{\tau}_h) &  \forall \bm \tau_h \in \bm V_h, \\
    \bm \sigma_h^0 = \bm \sigma_{0,h}. & 
\end{cases}
\end{equation}

\begin{theorem}\label{thm:stability_discrete}
Let $\bm \sigma_h^n \in \bm V_h$ be the solution of \eqref{eq:semi_discrete_partial} for $n=1,..., N_T$. Then, it holds
\begin{align*}
     & \max_{n \in \{1,,..,N_T\}} \left\|{\mu^{-\frac12}\rm dev}(\boldsymbol{\sigma}_h^n)\right\|^2_{\mathbb{L}^{2}(\Omega)} +  \Delta t \sum_{\ell=1}^{N} | \bm \sigma^\ell_h|_{dG}^2  \lesssim
      \left\|\mu^{-\frac12}\boldsymbol{{\rm dev}} (\boldsymbol{\sigma}_{0,h})\right\|^2_{\mathbb{L}^{2}(\Omega)} + \Delta t \sum_{\ell=1}^{N} \ \left\|\boldsymbol{F}^\ell\right\|^2_{\mathbb{L}^{2}(\Omega)}.
\end{align*}
\end{theorem} 
\begin{proof}
We consider $\boldsymbol{\tau}_h = \boldsymbol{\sigma}_h^n$ in \eqref{eq:semi_discrete_partial} and sum over $1 \le \ell \le k $  to obtain
\begin{align}\label{eq:fully_discrete_partial}
\sum_{\ell=1}^{k}\mathcal{M}(\bm{\sigma}_h^\ell -  \bm{\sigma}_h^{\ell-1}, \boldsymbol{\sigma}_h^\ell ) + \Delta t \sum_{\ell=1}^{k} \mathcal{A}(\boldsymbol{\sigma}_h^\ell, \boldsymbol{\sigma}_h^\ell) = \Delta t \sum_{\ell=1}^{k} ( \boldsymbol{F}^\ell, \boldsymbol{\sigma}_h^\ell).
\end{align}
%
%
For the first term on the left-hand side, we use that $A(A-B) = \frac{1}{2} (A^2 + (A-B)^2 - B^2)$ to get 
\begin{equation*}
  \sum_{\ell=1}^{k}\mathcal{M}(\bm{\sigma}_h^\ell -  \bm{\sigma}_h^{\ell-1}, \boldsymbol{\sigma}_h^\ell ) \geq  \frac12\mathcal{M}(\bm{\sigma}_h^k, \boldsymbol{\sigma}_h^k )  - \frac12\mathcal{M}(\bm{\sigma}_{0,h}, \boldsymbol{\sigma}_{0,h}). 
\end{equation*}
Next, we use the above inequality together with 
the Cauchy-Schwarz and Young inequalities on the right-hand side of \eqref{eq:fully_discrete_partial} to obtain
\begin{align*}
&    \frac{1}{2} \left\|\mu^{-\frac12}\boldsymbol{{\rm dev}} (\boldsymbol{\sigma}_h^k)\right\|^2_{\mathbb{L}^2(\Omega)} +  \Delta t \sum_{\ell=1}^{k} \ \mathcal{A}(\boldsymbol{\sigma}_h^\ell, \boldsymbol{\sigma}_h^\ell) \\
    & \qquad \qquad \lesssim \frac{1}{2} \left\|\mu^{-\frac12}\boldsymbol{{\rm dev}} (\boldsymbol{\sigma}_{0,h})\right\|^2_{\mathbb{L}^2(\Omega)} + \frac{\Delta t}{2\epsilon} \sum_{\ell=1}^{k}  \ \left\|\boldsymbol{F}^\ell\right\|^2_{\mathbb{L}^2(\Omega)} + \frac{\epsilon \Delta t }{2} \sum_{\ell=1}^{k} \ \left\|\boldsymbol{\sigma}_h^\ell\right\|^2_{\mathbb{L}^2(\Omega)},
\end{align*}
for $\epsilon>0$. Then, by using the coercivity of $\mathcal{A}(\cdot, \cdot)$ in \eqref{eq:coercivity_dg} it holds
\begin{align}
& \left\|\mu^{-\frac12}\boldsymbol{{\rm dev}} (\boldsymbol{\sigma}_h^k)\right\|^2_{\mathbb{L}^2(\Omega)} + \Delta t \sum_{\ell=1}^{k} | \bm \sigma^\ell_h|_{dG}^2 
    \nonumber \\ & \qquad \qquad \lesssim   \left\|\mu^{-\frac12}\boldsymbol{{\rm dev}} (\boldsymbol{\sigma}_{0,h})\right\|^2_{\mathbb{L}^2(\Omega)} + \frac{\Delta t}{\epsilon} \sum_{\ell=1}^{k} \ \left\|\boldsymbol{F}^\ell\right\|^2_{\mathbb{L}^2(\Omega)} + \epsilon \Delta t \sum_{\ell=1}^{k} \ \left\|\boldsymbol{\sigma}_h^\ell\right\|^2_{\mathbb{L}^2(\Omega)}.
\end{align}
Next, we apply Lemma~\ref{lemma:sigma_l2} to estimate the last term on the right-hand side, recalling that $k \Delta t = t_{k}$, and we obtain:
\begin{align*}
 \sum_{\ell=1}^{k} \left\|\boldsymbol{\sigma}_h^\ell\right\|_{\mathbb{L}^{2}(\Omega)}^{2} & \lesssim  \sum_{\ell=1}^{k}\Big(  \left\|\mu^{-\frac12}\boldsymbol{{\rm dev}}(\boldsymbol{\sigma}_h^\ell)\right\|_{\mathbb{L}^{2}(\Omega)}^{2} + | \bm \sigma_h^\ell|^2_{dG} \Big) \\
 & \leq  \Big( k  \max_{\ell \in \{1,...,k\}} \left\|\mu^{-\frac12}\boldsymbol{{\rm dev}}(\boldsymbol{\sigma}_h^\ell)\right\|_{\mathbb{L}^{2}(\Omega)}^{2} +  \sum_{\ell=1}^{k}  | \bm \sigma_h^\ell|^2_{dG} \Big),
 \end{align*}
 so that 
 \begin{align*}
 \Delta t \sum_{\ell=1}^{k} \left\|\boldsymbol{\sigma}_h^\ell\right\|_{\mathbb{L}^{2}(\Omega)}^{2}     
& \lesssim  T\max_{\ell \in \{1,...,k\}} \left\|\mu^{-\frac12}\boldsymbol{{\rm dev}}(\boldsymbol{\sigma}_h^\ell)\right\|_{\mathbb{L}^{2}(\Omega)}^{2} + \Delta t\sum_{\ell=1}^{k}  | \bm \sigma_h^\ell|^2_{dG}.
\end{align*}
\\
Then, we have that 
\begin{align*}
& \left\|\mu^{-\frac12}\boldsymbol{{\rm dev}} (\boldsymbol{\sigma}_h^k)\right\|^2_{\mathbb{L}^2(\Omega)} + \Delta t \sum_{\ell=1}^{k} | \bm \sigma^\ell_h|_{dG}^2 
    \nonumber 
    \lesssim   \left\|\mu^{-\frac12}\boldsymbol{{\rm dev}} (\boldsymbol{\sigma}_{0,h})\right\|^2_{\mathbb{L}^2(\Omega)} + \frac{\Delta t}{\epsilon} \sum_{\ell=1}^{k} \ \left\|\boldsymbol{F}^\ell\right\|^2_{\mathbb{L}^2(\Omega)} \\ & \qquad \qquad  + \epsilon  \Big( T \max_{\ell \in \{1,...,k\}} \left\|\mu^{-\frac12}\boldsymbol{{\rm dev}}(\boldsymbol{\sigma}_h^\ell)\right\|_{\mathbb{L}^{2}(\Omega)}^{2} + \Delta t \sum_{\ell=1}^{k}  | \bm \sigma_h^\ell|^2_{dG} \Big).
\end{align*}
%
%
The theorem follows by taking the maximum over $k \in \{1,...,N_T\}$ and $\epsilon$ small enough.
%
\end{proof}

\begin{remark}\label{rem:tetametodo}
To obtain the stability estimate as in Theorem \ref{thm:stability_discrete} for the general case $\theta \in [\frac12, 1)$, one has to consider the following equation
\begin{multline*}
    \mathcal{M}(\bm{\sigma}_h^n -  \bm{\sigma}_h^{n-1}, \bm \tau_h )    + \Delta t  \mathcal{A}(\theta \bm \sigma^n_h + (1-\theta) \bm \sigma^{n-1}_h,\bm \tau_h)  = \Delta t  ( \theta \boldsymbol{F}^n + (1-\theta) \boldsymbol{F}^{n-1},\bm \tau_h) ,
\end{multline*}
for $n=1,..., N_T$. Taking $\bm \tau_h = \theta \bm \sigma^n_h + (1-\theta) \bm \sigma^{n-1}_h$, we observe that 
\begin{align*}
    \mathcal{M}(\bm{\sigma}_h^n -  \bm{\sigma}_h^{n-1}, \theta \bm \sigma^n_h + (1-\theta) \bm \sigma^{n-1}_h  ) & = \frac12\mathcal{M}(\bm{\sigma}_h^n, \bm \sigma^n_h) - \frac12\mathcal{M}(\bm{\sigma}_h^{n-1}, \bm \sigma^{n-1}_h)     \\ & \quad + (\theta -\frac12)\mathcal{M}(\bm{\sigma}_h^n-\bm{\sigma}_h^{n-1}, \bm{\sigma}_h^n-    \bm \sigma^{n-1}_h)  \\
    & \geq \frac12\mathcal{M}(\bm{\sigma}_h^n, \bm \sigma^n_h) - \frac12\mathcal{M}(\bm{\sigma}_h^{n-1}, \bm \sigma^{n-1}_h),
\end{align*}
for $\theta \geq \frac12$, use the coercivity of $\mathcal{A}(\cdot,\cdot)$, and proceed as in the proof of Theorem~\ref{thm:stability_discrete}.
\end{remark}

\section{Error analysis of the fully-discrete problem} \label{sec:convergence}
Before presenting the main results of this section, we introduce some preliminary results that are instrumental for the proof of the \textit{a priori} error estimates  \cite{Antonietti2021,Botti2020_korn}.
Let  $\mathcal{E}: \boldsymbol{H}^m(\kappa) \rightarrow \boldsymbol{H}^m(\mathbb{R}^d)$ be the Stein extension operator s.t. for any $\kappa \in \mathcal{T}_h$ and $m \in \mathbb{N}_0$,  $\mathcal{E} \boldsymbol{\sigma} |_\kappa = \boldsymbol{\sigma}$ and $\left\|\mathcal{E} \boldsymbol{\sigma}\right\|_{\mathbb{H}^m(\mathbb{R}^{d\times d})} \lesssim \left\| \boldsymbol{\sigma}\right\|_{ \mathbb{H}^m(\kappa)} $.
Then, for any $\boldsymbol{\sigma} \in \mathbb{H}^{m}(\mathcal{T}_h)$, with $m \ge 2$, there exists $\boldsymbol{\pi} \boldsymbol{\sigma} \in \boldsymbol{V}_h$ such that
\begin{equation}\label{eq:interp_sigma}
    \left\| \boldsymbol{\sigma} - \boldsymbol{\pi \sigma}\right\|_{\mathbb{L}^2(\Omega)}^2 \lesssim \sum_{\kappa\in \mathcal{T}_h} \frac{h_k^{2s_k}}{p_k^{2m}} \left\|\mathcal{E}\boldsymbol{\sigma}\right\|^2_{\mathbb{H}^{m} (\kappa)}, 
\end{equation}

\begin{equation} \label{eq:estimate_sigma}  \trinorm{\boldsymbol{\sigma} - \boldsymbol{\pi \sigma}}_{dG}^2  \lesssim \sum_{\kappa\in \mathcal{T}_h} \frac{h_k^{2s_k-2}}{p_k^{2m-3}} \left\|\mathcal{E}\boldsymbol{\sigma}\right\|^2_{\mathbb{H}^{m} (\kappa)},
\end{equation}
where $s_k = \min\{p_k+1,m\}$. 
%
\begin{lemma}\label{lemma:dev_bound}
Under Assumptions \ref{ass::regular}-\ref{ass::3}  we suppose that the solution $\boldsymbol{\sigma}$ to (\ref{eq:eqnn}) is regular enough, i.e. $\boldsymbol{\sigma} \in C^1((0,T];\boldsymbol{H}^{m}(\mathcal{T}_h))\cap C^2((0,T];\bm L^2(\Omega))$, with $m \ge 2$. Then, for any $t_n = n\Delta t$ with $n = 1,...,N_T$, it holds  
\begin{align*}
 \left\|\mu^{-\frac12}\boldsymbol{{\rm dev}}(\dot{\bm \sigma}^n -\bm \pi \dot{\bm \sigma}^n) \right\|^2_{\mathbb{L}^2(\Omega)} & \lesssim  \sum_{\kappa\in \mathcal{T}_h} \frac{h_k^{2s_k}}{p_k^{2m}} \left\|\mathcal{E}\dot{\boldsymbol{\sigma}}^n\right\|^2_{\mathbb{H}^{m} (\kappa)},
 \\ 
  \max_{n \in \{1,...,N_T\}} \left\| \mu^{-\frac12}\boldsymbol{{\rm dev}}\left(\bm \pi \dot{\bm\sigma}^n - \frac{\bm \pi \bm\sigma^n - \bm \pi \bm\sigma^{n-1}}{\Delta t }\right)\right\|^2_{\mathbb{L}^2(\Omega)} & \lesssim \Delta t^2  \left\|\bm\sigma \right\|^2_{C^2((0,T];\bm L^2(\Omega))},
\end{align*} 
where $\bm \pi$ is defined as before.
\end{lemma}
\begin{proof}
The first inequality follows from the definition of the $\boldsymbol{{\rm dev}}$ operator and \eqref{eq:interp_sigma}, i.e.,
\begin{equation*}
    \left\|\mu^{-\frac12}(\boldsymbol{{\rm dev}}(\dot{\bm \sigma}^n) - \boldsymbol{{\rm dev}}(\bm \pi \dot{\bm \sigma}^n) ) \right\|^2_{\mathbb{L}^2(\Omega)} \lesssim \left\|\mu^{-\frac12}(\dot{\bm \sigma}^n - \bm \pi \dot{\bm \sigma}^n)\right\|^2_{\mathbb{L}^2(\Omega)}.
\end{equation*}
For the second inequality, we employ again the definition of $\boldsymbol{{\rm dev}}$ and twice the Lagrange theorem to obtain the existence of $\xi,\zeta \in (t_{n-1},t_n)$ such that
\begin{align*}
 \left\| \mu^{-\frac12}\boldsymbol{{\rm dev}}\left(\bm \pi \dot{\bm\sigma}^n - \frac{\bm \pi \bm\sigma^n - \bm \pi \bm\sigma^{n-1}}{\Delta t }\right)\right\|^2_{\mathbb{L}^2(\Omega)} & \lesssim  
 \left\| \mu^{-\frac12}(\bm \pi \dot{\bm\sigma}(t_n) - \bm \pi \dot{\bm\sigma}(\xi))\right\|^2_{\mathbb{L}^2(\Omega)}  \lesssim \Delta t^2 \left\| \mu^{-\frac12} \bm \pi \ddot{\bm\sigma}(\zeta) \right\|^2_{\mathbb{L}^2(\Omega)}. 
\end{align*}
Owing to the regularity assumption on the exact solution $\bm\sigma$, and taking into account the boundedness of the projection operator $\bm \pi$ in $\bm L^2(\Omega)$, this concludes the proof.
\end{proof}

We observe that for any $n=1,...,N_T$,  $\boldsymbol{\sigma}^n \in \mathbb{H}({\rm div},\Omega)$ solves the following problem:
 \begin{equation}\label{weakform1}
 \mathcal{M}(\dot{\boldsymbol{\sigma}}^n,\boldsymbol{\tau}_h) + \mathcal{A}(\boldsymbol{\sigma}^n, \boldsymbol{\tau}_h) = ( \boldsymbol{F}^n , \boldsymbol{\tau}_h )_\Omega \quad \boldsymbol{\tau}_h \in \boldsymbol{V}_h.
\end{equation}
Then, we  consider the error equation obtained by subtracting \eqref{weakform1} from \eqref{eq:semi_discrete_partial}:
\begin{align*}
 \mathcal{M}( \bm\sigma_h^n - \bm\sigma_h^{n-1},\boldsymbol{\tau}_h) -  \Delta t \mathcal{M}( \dot{\bm\sigma}^n ,\boldsymbol{\tau}_h) + \Delta t\mathcal{A}(\boldsymbol{e}^n, \boldsymbol{\tau}_h) = 0,
\end{align*}
where we have set $\bm e^n = \boldsymbol{\sigma}_h^n - \boldsymbol{\sigma}^n$.
By adding and subtracting the term $\mathcal{M}( \bm\sigma^n - \bm\sigma^{n-1},\boldsymbol{\tau}_h)$, we can rewrite the above equation as
\begin{align}\label{eq:error_equation}
 \mathcal{M}( \bm e^n - \bm e^{n-1},\boldsymbol{\tau}_h)  + \Delta t\mathcal{A}(\boldsymbol{e}^n, \boldsymbol{\tau}_h) =   \Delta t \mathcal{M}( \dot{\bm\sigma}^n ,\boldsymbol{\tau}_h) - \mathcal{M}( \bm\sigma^n - \bm\sigma^{n-1},\boldsymbol{\tau}_h),
\end{align}
for all $\bm \tau_h \in \bm V_h$. 
In the following, we use the notation
\begin{equation*}
  \| \bm e \|_{E}^2 =  \max_{n=\{1,...,N_T\}}\left\| \mu^{-\frac12}\boldsymbol{{\rm dev}} (\boldsymbol{e}^n) \right\|^2_{\mathbb{L}^2(\Omega)} + \Delta t  \sum_{n=1}^{N_T} |\bm e^n|_{dG}^2, 
\end{equation*}
    with $\bm e(t) = \bm \sigma(t) - \bm \sigma_h(t)$ and $\bm e^n = \bm \sigma(t_n) - \bm \sigma_h^n$.
We next state the main result of this section.

\begin{theorem}\label{thm:convergence}
Let $\boldsymbol{\sigma} \in C^1((0,T];\boldsymbol{H}^{m}(\mathcal{T}_h))\cap C^2((0,T];\bm L^2(\Omega))$, with $m\geq2$, be the solution of \eqref{weakform1}, and let $\bm \sigma_h^n \in \bm V_h$, for $n=1,...,N_T$, be the solution of \eqref{eq:semi_discrete_partial} for a sufficiently large $\alpha$. Then, under Assumptions \ref{ass::regular}-\ref{ass::3} it holds 
\begin{align*}
  & 
   \|  \bm \sigma - \bm \sigma_h \|_{E}^2 
    \lesssim 
\max_{n \in \{1,...,N_T\}} \Big( \sum_{\kappa\in \mathcal{T}_h} \frac{h_k^{2s_k-2}}{p_k^{2m-3}} \left\|\mathcal{E}\boldsymbol{\sigma}^n\right\|^2_{\mathbb{H}^{m} (\kappa)}
   + \sum_{\kappa\in \mathcal{T}_h} \frac{h_k^{2s_k}}{p_k^{2m}} \left\|\mathcal{E}\dot{\boldsymbol{\sigma}^n}\right\|^2_{\mathbb{H}^{m} (\kappa)} \\ &\qquad\qquad\qquad\qquad \qquad\qquad +  \sum_{\kappa\in \mathcal{T}_h} \frac{h_k^{2s_k}}{p_k^{2m}} \left\|\mathcal{E}{\boldsymbol{\sigma}^n}\right\|^2_{\mathbb{H}^{m} (\kappa)} \Big)  +
     \Delta t^2 \left\|\bm\sigma \right\|^2_{C^2((0,T];\bm L^2(\Omega))}.
\end{align*}
\end{theorem}

\begin{proof}
For any $n=1,...,N_T$ we split the error in the following way: $\boldsymbol{e}^n = \boldsymbol{e}_h^n - \boldsymbol{e}_I^n$, where
$   \boldsymbol{e}_h^n = \boldsymbol{\sigma}_h^n - \boldsymbol{\pi} \boldsymbol{\sigma}^n$ and $
   \boldsymbol{e}_I^n = \boldsymbol{\sigma}^n - \boldsymbol{\pi} \boldsymbol{\sigma}^n$.
Next, by Young's inequality we get 
\begin{equation}\label{eq:triangle_thm}
     \|  \bm \sigma - \bm \sigma_h \|_{E}^2\lesssim  \|  \bm \sigma - \bm \pi \bm \sigma \|_{E}^2  +  \|  \bm \sigma_h - \bm \pi \bm \sigma \|_{E}^2.
\end{equation}
To bound the first term above we use both \eqref{eq:interp_sigma} and \eqref{eq:estimate_sigma}, obtaining
\begin{align}\label{eq:estimate_pi}
    \|  \bm \sigma - \bm \pi \bm \sigma \|_{E}^2 \lesssim  
    \max_{n \in \{1,...,N_T\}} \Big( \sum_{\kappa\in \mathcal{T}_h} \frac{h_k^{2s_k}}{p_k^{2m}} \left\|\mathcal{E}\boldsymbol{\sigma}^n\right\|^2_{\mathbb{H}^{m} (\kappa)}
   + \sum_{\kappa\in \mathcal{T}_h} \frac{h_k^{2 s_k-2}}{p_k^{2m-3}} \left\|\mathcal{E}{\boldsymbol{\sigma}^n}\right\|^2_{\mathbb{H}^{m} (\kappa)} \Big).
\end{align}
For the second term in \eqref{eq:triangle_thm} we 
choose $\boldsymbol{\tau}_h = \boldsymbol{e}_h^n$ in \eqref{eq:error_equation} and obtain:
\begin{align*}
 \mathcal{M}( \bm e_h^n - \bm e_h^{n-1},\boldsymbol{e}_h^n)  + \Delta t\mathcal{A}(\boldsymbol{e}_h^n, \boldsymbol{e}_h^n) & =   \Delta t \mathcal{M}( \dot{\bm\sigma}^n ,\boldsymbol{e}_h^n) - \mathcal{M}( \bm\sigma^n - \bm\sigma^{n-1},\boldsymbol{e}_h^n)\\ & 
 \qquad + \mathcal{M}( \bm e_I^n - \bm e_I^{n-1},\boldsymbol{e}_h^n)  + \Delta t\mathcal{A}(\boldsymbol{e}_I^n, \boldsymbol{e}_h^n)
\end{align*}
that is 
\begin{align*}
 \mathcal{M}( \bm e_h^n - \bm e_h^{n-1},\boldsymbol{e}_h^n)  + \Delta t\mathcal{A}(\boldsymbol{e}_h^n, \boldsymbol{e}_h^n) & =   \Delta t \mathcal{M}( \dot{\bm\sigma}^n ,\boldsymbol{e}_h^n) - \mathcal{M}(\bm \pi \bm\sigma^n - \bm \pi \bm\sigma^{n-1},\boldsymbol{e}_h^n) + \Delta t\mathcal{A}(\boldsymbol{e}_I^n, \boldsymbol{e}_h^n).
\end{align*}
Next, we add and subtract the term $\Delta t\mathcal{M}(\bm \pi \dot{\bm\sigma}^n,\boldsymbol{e}_h^n)$ and sum over $1 \le \ell \le k $,  to get
\begin{align*}
\sum_{\ell=1}^{k} \Big( \mathcal{M}( \bm e_h^\ell - \bm e_h^{\ell-1},\boldsymbol{e}_h^\ell)  + \Delta t\mathcal{A}(\boldsymbol{e}_h^\ell, \boldsymbol{e}_h^\ell)\Big) & = \sum_{\ell=1}^{k} \Big(  \Delta t\mathcal{A}(\boldsymbol{e}_I^\ell, \boldsymbol{e}_h^\ell) + \Delta t \mathcal{M}( \dot{\bm e}_I^\ell,\boldsymbol{e}_h^\ell) \\ & \qquad + \mathcal{M}(\Delta t\bm \pi \dot{\bm\sigma}^\ell - (\bm \pi \bm\sigma^\ell - \bm \pi \bm\sigma^{\ell-1}),\boldsymbol{e}_h^\ell)\Big).
\end{align*}
As in the proof of Theorem \ref{thm:stability_continuous} we use  that 
$A(A-B) = \frac{1}{2} (A^2 + (A-B)^2 - B^2)$, together with the coercivity of $\mathcal{A}(\cdot, \cdot)$ in \eqref{eq:coercivity_dg} with $\alpha$ large enough, obtaining:
\begin{align*}
    & \left\| \mu^{-\frac12}\boldsymbol{{\rm dev}} (\boldsymbol{e}_h^k) \right\|^2_{\mathbb{L}^2(\Omega)} + \Delta t  \sum_{\ell=1}^{k} |\bm e_h^\ell|_{dG}^2 \\ & \qquad \lesssim
\Delta t \sum_{\ell=1}^{k} \Big(  \mathcal{A}(\boldsymbol{e}_I^\ell, \boldsymbol{e}_h^\ell) +  \mathcal{M}( \dot{\bm e}_I^\ell,\boldsymbol{e}_h^\ell) +  \mathcal{M}( \bm \pi \dot{\bm\sigma}^\ell - \frac{(\bm \pi \bm\sigma^\ell - \bm \pi \bm\sigma^{\ell-1})}{\Delta t },\boldsymbol{e}_h^\ell)\Big)
\end{align*}
noticing that $\boldsymbol{e}_h^0 = \boldsymbol{0}$. 
We employ now \eqref{eq:continuity_dg} together with  Cauchy–Schwarz and Young inequalities to obtain:
\begin{align*}
  &  \left\| \mu^{-\frac12}\boldsymbol{{\rm dev}} (\boldsymbol{e}_h^k) \right\|^2_{\mathbb{L}^2(\Omega)} + \Delta t  \sum_{\ell=1}^{k} |\bm e_h^\ell|_{dG}^2 \\
 &  \lesssim
\Delta t \sum_{\ell=1}^{k} \Big( \frac{1}{\epsilon_1}\||\boldsymbol{e}_I^\ell|\|_{dG}^2  + \epsilon_1 |\boldsymbol{e}_h^\ell|_{dG}^2 +  \frac{1}{\epsilon_2}\| \mu^{-\frac12}\boldsymbol{{\rm dev}} (\dot{\bm e}_I^\ell) \|^2 + \epsilon_2 \|\mu^{-\frac12}\boldsymbol{{\rm dev}} (\boldsymbol{e}_h^\ell)\|^2 \\ & \qquad \qquad + \frac{1}{\epsilon_3}\| \mu^{-\frac12}\boldsymbol{{\rm dev}}(\bm \pi \dot{\bm\sigma}^\ell - \frac{\bm \pi \bm\sigma^\ell - \bm \pi \bm\sigma^{\ell-1}}{\Delta t })\|^2 +\epsilon_3 \|\mu^{-\frac12} \boldsymbol{{\rm dev}}(\boldsymbol{e}_h^\ell) \|^2 \Big)  \\
& \lesssim 
\Delta t \sum_{\ell=1}^{k} \Big( \frac{1}{\epsilon_1}\||\boldsymbol{e}_I^\ell|\|_{dG}^2   +  \frac{1}{\epsilon_2}\| \mu^{-\frac12}\boldsymbol{{\rm dev}} (\dot{\bm e}_I^\ell) \|^2 + \frac{1}{\epsilon_3}\| \mu^{-\frac12}\boldsymbol{{\rm dev}}(\bm \pi \dot{\bm\sigma}^\ell - \frac{\bm \pi \bm\sigma^\ell - \bm \pi \bm\sigma^{\ell-1}}{\Delta t })\|^2
 \\ & \qquad \qquad  + \epsilon_1 |\boldsymbol{e}_h^\ell|_{dG}^2 \Big) + (\epsilon_2+\epsilon_3) k \Delta t \max_{\ell\in \{ 1,...,n\}}\|\mu^{-\frac12} \boldsymbol{{\rm dev}}(\boldsymbol{e}_h^\ell) \|^2  
\end{align*}
Next, by taking $\epsilon_1$ small enough, noticing that $k\Delta t = t_k \leq T$, choosing $\epsilon_2$ and $\epsilon_3$ small enough,  
and taking the maximun for $k \in \{1,...,N_T\}$ we have
\begin{align*}
  & \max_{k=\{1,...,N_T\}}\left\| \mu^{-\frac12}\boldsymbol{{\rm dev}} (\boldsymbol{e}_h^k) \right\|^2_{\mathbb{L}^2(\Omega)} + \Delta t  \sum_{\ell=1}^{N_T} |\bm e_h^\ell|_{dG}^2 \\
 &  \lesssim 
\Delta t \sum_{\ell=1}^{N_T} \Big( \||\boldsymbol{e}_I^\ell|\|_{dG}^2   +  \| \mu^{-\frac12}\boldsymbol{{\rm dev}} (\dot{\bm e}_I^\ell) \|^2 + \| \mu^{-\frac12}\boldsymbol{{\rm dev}}(\bm \pi \dot{\bm\sigma}^\ell - \frac{\bm \pi \bm\sigma^\ell - \bm \pi \bm\sigma^{\ell-1}}{\Delta t })\|^2 \Big) \\
& \lesssim 
T \max_{\ell \in \{1,...,N_T\}} \Big( \||\boldsymbol{e}_I^\ell|\|_{dG}^2   +  \| \mu^{-\frac12}\boldsymbol{{\rm dev}} (\dot{\bm e}_I^\ell) \|^2 + \| \mu^{-\frac12}\boldsymbol{{\rm dev}}(\bm \pi \dot{\bm\sigma}^\ell - \frac{\bm \pi \bm\sigma^\ell - \bm \pi \bm\sigma^{\ell-1}}{\Delta t })\|^2 \Big). 
\end{align*}
Finally, we use the result in Lemma~\ref{lemma:dev_bound} and the estimates \eqref{eq:interp_sigma}, \eqref{eq:estimate_sigma}, and we obtain
\begin{align*}
  & \max_{k=\{1,...,N_T\}}\left\| \mu^{-\frac12}\boldsymbol{{\rm dev}} (\boldsymbol{e}_h^k) \right\|^2_{\mathbb{L}^2(\Omega)} +   \sum_{\ell=1}^{N_T} \Delta t |\bm e_h^\ell|_{dG}^2 \\
& \lesssim 
T \max_{\ell \in \{1,...,N_T\}} \Big( \sum_{\kappa\in \mathcal{T}_h} \frac{h_k^{2(s_k-1)}}{p_k^{2(m-3/2)}} \left\|\mathcal{E}\boldsymbol{\sigma}^\ell\right\|^2_{\mathbb{H}^{m} (\kappa)}
   + \sum_{\kappa\in \mathcal{T}_h} \frac{h_k^{2s_k}}{p_k^{2m}} \left\|\mathcal{E}\dot{\boldsymbol{\sigma}}^\ell\right\|^2_{\mathbb{H}^{m} (\kappa)} \Big) +  
     \Delta t^2 \left\|\bm\sigma \right\|^2_{C^2((0,T];\bm L^2(\Omega))}.
\end{align*}

By combining the last inequality to \eqref{eq:estimate_pi} we conclude the proof.
\end{proof}

\section{Numerical results}\label{sec:numerical_test}

\noindent In this section, we present the results of the numerical simulations obtained through the discretization in Section~\ref{sec:dG_discretization} implemented in the open source MATLAB library \\ \texttt{lymph} (\texttt{https://bitbucket.org/lymph/lymph/src/StokesPS}) \cite{lymph2024}. 
First, we consider a test case to demonstrate the theoretical error bounds. Next, we investigate the capability of the proposed approach to recover the pressure and the velocity fields of the velocity-pressure Stokes formulation, as explained in Remark~\ref{rem-sub:p_and_u_rec}. \textcolor{black}{Then we solve a test case of engineering interest, namely, the flow around a circular cylinder, and finally we propose two 3D test cases.}

\subsection{Verification test case}\label{sec:unsteady}
In the following simulations, we consider the domain $\Omega = (0,1)^2 $, and four different polygonal meshes having different granularity, cf. Figure~\ref{fig:1}. Dirichlet boundary conditions are imposed on the top and right boundaries (blue lines), while Neumann conditions are applied on the remaining part of the boundary (red lines). We also set  $\mu = 1$.
\begin{figure}
    \centering
   \subfloat[]
        {\includegraphics[scale=0.3]{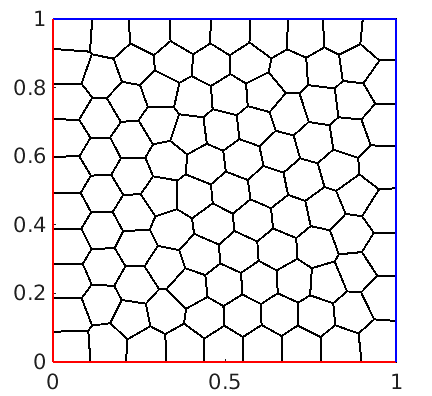}}
    \subfloat[]
        {\includegraphics[scale=0.3]{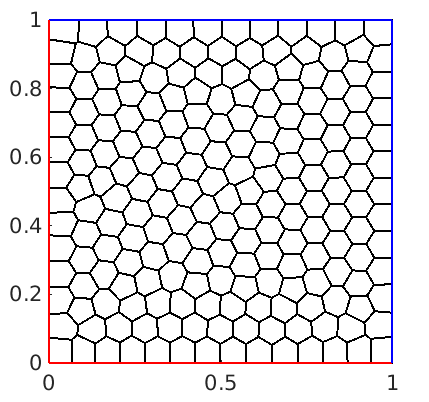}}
    \subfloat[]
        {\includegraphics[scale=0.3]{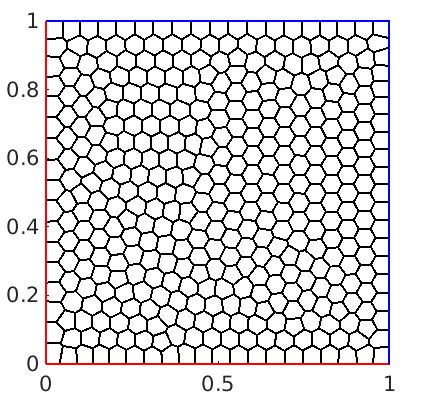}}
    \subfloat[]
        {\includegraphics[scale=0.3]{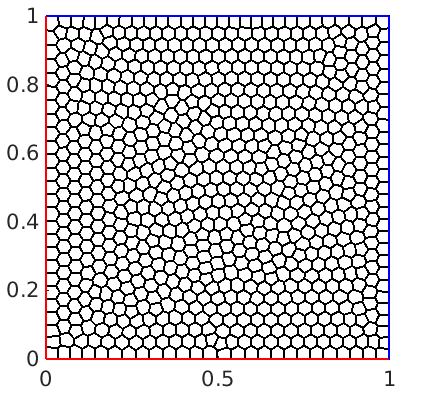}}
    \caption{Test case of Section~\ref{sec:unsteady}. Considered Polytopal meshes: (a) 100 elements ($h \approx 0.1759$), (b) 200 elements ($h \approx 0.1260$), (c) 400 elements ($h \approx 0.0909$), and (d) 800 elements ($h \approx 0.0637$). Boundary edges are highlighted in blue for Dirichlet conditions, and red for Neumann conditions.}
    \label{fig:1}
\end{figure}
We consider the following exact solution of problem \eqref{eq:eqnn}:
$$
\boldsymbol{\sigma}(\bm x,t) = \sin(2t) \begin{bmatrix}\sin(\pi x)\sin(\pi y) & 0 \\ 0 & -\sin(\pi x)\sin(\pi y)\end{bmatrix}.        
$$
\begin{figure}
    \centering
%
%
\begin{tikzpicture}

\begin{axis}[%
width=0.35\textwidth,
height=0.35\textwidth,
scale only axis,
xmode=log,
xmin=0.05,
xmax=0.2,
xminorticks=true,
xlabel style={font=\color{black}},
xlabel={$h$},
ymode=log,
ymin=1.e-8,
ymax=0.2,
ylabel style={font=\color{black}},
ylabel={$\| \bm \sigma - \bm \sigma_h \|_{E}$},
axis background/.style={fill=white},
yminorticks=true,
xmajorgrids,
xminorgrids,
ymajorgrids,
yminorgrids,
legend style={at={(0.7,0.000)}, anchor=south west, legend cell align=left, align=left, draw=white!15!black}
]
\addplot [color=blue, line width=2.0pt]
  table[row sep=crcr]{%
0.17589951497685	0.0644449217769595\\
0.126032381980115	0.0428043856286823\\
0.0909826836096301	0.0283725688147399\\
0.0636997815814856	0.0188496863851049\\
};
\addlegendentry{$p = 1$}

\addplot [color=black, forget plot]
  table[row sep=crcr]{%
0.0636998	0.015\\
0.0909827	0.0214\\
};

\addplot [color=black, forget plot]
  table[row sep=crcr]{%
0.0909827	0.015\\
0.0909827	0.0214\\
};
\addplot [color=black, forget plot]
  table[row sep=crcr]{%
0.0636998	0.015\\
0.0909827	0.015\\
};

\node[right, align=left, text=black, font=\normalsize]
at (axis cs:0.09,0.017) {1};

\addplot [color=blue, line width=2.0pt, mark=o, mark options={solid, blue}]
  table[row sep=crcr]{%
0.17589951497685	0.007959154874184\\
0.126032381980115	0.00394439022644889\\
0.0909826836096301	0.00197989694427489\\
0.0636997815814856	0.000988609931842377\\
};
\addlegendentry{$p = 2$}

\addplot [color=black, forget plot]
  table[row sep=crcr]{%
0.0636998	0.0005\\
0.0909827	0.001\\
};

\addplot [color=black, forget plot]
  table[row sep=crcr]{%
0.0909827	0.0005\\
0.0909827	0.001\\
};
\addplot [color=black, forget plot]
  table[row sep=crcr]{%
0.0636998	0.0005\\
0.0909827	0.0005\\
};

\node[right, align=left, text=black, font=\normalsize]
at (axis cs:0.09,0.00075) {2};

\addplot [color=blue, line width=2.0pt, mark=square, mark options={solid, blue}]
  table[row sep=crcr]{%
0.17589951497685	0.000480495273951811\\
0.126032381980115	0.000164930122260953\\
0.0909826836096301	5.55340441423886e-05\\
0.0636997815814856	1.90360174416475e-05\\
};
\addlegendentry{$p = 3$}

\addplot [color=black, forget plot]
  table[row sep=crcr]{%
0.0636998	1.e-05\\
0.0909827	3.e-05\\
};

\addplot [color=black, forget plot]
  table[row sep=crcr]{%
0.0909827	1.e-05\\
0.0909827	3.e-05\\
};
\addplot [color=black, forget plot]
  table[row sep=crcr]{%
0.0636998	1.e-05\\
0.0909827	1.e-05\\
};

\node[right, align=left, text=black, font=\normalsize]
at (axis cs:0.09,1.75e-5) {3};

\addplot [color=blue, line width=2.0pt, mark=diamond, mark options={solid, blue}]
  table[row sep=crcr]{%
0.17589951497685	1.79535882589395e-05\\
0.126032381980115	4.34334434419561e-06\\
0.0909826836096301	1.0719671133024e-06\\
0.0636997815814856	2.81599264050311e-07\\
};
\addlegendentry{$p = 4$}

\addplot [color=black, forget plot]
  table[row sep=crcr]{%
0.0636998	1.5e-07\\
0.0909827	6.e-07\\
};

\addplot [color=black, forget plot]
  table[row sep=crcr]{%
0.0909827	1.5e-07\\
0.0909827	6.e-07\\
};
\addplot [color=black, forget plot]
  table[row sep=crcr]{%
0.0636998	1.5e-07\\
0.0909827	1.5e-07\\
};

\node[right, align=left, text=black, font=\normalsize]
at (axis cs:0.09,3.e-7) {4};

\end{axis}
\end{tikzpicture}%
%
%
\definecolor{mycolor1}{rgb}{0.00000,0.44700,0.74100}%
\begin{tikzpicture}

\begin{axis}[%
width=0.35\textwidth,
height=0.35\textwidth,
scale only axis,
scale only axis,
xmode=log,
xmin=0.01,
xmax=0.15,
xminorticks=true,
xlabel style={font=\color{black}},
xlabel={$\Delta t$},
ymode=log,
ymin=1.e-8,
ymax=0.2,
yminorticks=true,
axis background/.style={fill=white},
xmajorgrids,
xminorgrids,
ymajorgrids,
yminorgrids,
legend style={at={(0.0,0.000)}, anchor=south west,legend cell align=left, align=left, draw=white!15!black}
]
\addplot [color=blue, line width=2.0pt]
  table[row sep=crcr]{%
0.1	0.111640801079277\\
0.05	0.05595571790468\\
0.025	0.0279824611371833\\
0.0125	0.0139949137206521\\
};
\addlegendentry{Impicit Euler}

\addplot [color=black, forget plot]
  table[row sep=crcr]{%
0.0125	0.01\\
0.025	0.0214\\
};
\addplot [color=black, forget plot]
  table[row sep=crcr]{%
0.025	0.01\\
0.025	0.0214\\
};
\addplot [color=black, forget plot]
  table[row sep=crcr]{%
0.0125	0.01\\
0.025	0.01\\
};

\node[right, align=left, text=black, font=\normalsize]
at (axis cs:0.025,0.017) {1};

\addplot [color=blue, line width=2.0pt, mark=o, mark options={solid, blue}]
  table[row sep=crcr]{%
0.1	0.0021312861109385\\
0.05	0.000532138799575681\\
0.025	0.000133198684066059\\
0.0125	3.43117770869249e-05\\
};
\addlegendentry{Crank-Nicolson}

\addplot [color=black, forget plot]
  table[row sep=crcr]{%
0.0125	2.e-5\\
0.025	8.e-5\\
};
\addplot [color=black, forget plot]
  table[row sep=crcr]{%
0.025	2.e-5\\
0.025	8.e-5\\
};
\addplot [color=black, forget plot]
  table[row sep=crcr]{%
0.0125	2.e-5\\
0.025	2.e-5\\
};

\node[right, align=left, text=black, font=\normalsize]
at (axis cs:0.025,4.e-5) {2};

\end{axis}
\end{tikzpicture}%
 \caption{Test case of Section~\ref{sec:unsteady}. Left: log-log plot of the computed error $\| \bm \sigma -\bm \sigma_h \|_{E}$ as a function of the mesh size $h$ for $p = 1,...,4$, by fixing $T=0.25$ and \textcolor{black}{$\Delta t = 1e-3$}. Right: log-log plot of the computed error $\| \bm \sigma -\bm \sigma_h \|_{E}$ as a function of time step $\Delta t$ fixing the polynomial degree $p = 4$, and $h \approx 0.0909$.}
    \label{fig:testcase2_convergence}
\end{figure}

\begin{figure}
\centering
\begin{subfigure}[b]{0.45\textwidth}
\centering
%
%
\definecolor{mycolor1}{rgb}{0.00000,0.44700,0.74100}%
\begin{tikzpicture}

\begin{axis}[%
width=0.4\textwidth,
height=0.4\textwidth,
scale only axis,
xmin=1,
xmax=6,
ymode=log,
ymin=1e-08,
ymax=0.1,
yminorticks=true,
xtick distance=1,
enlarge x limits=false,
xlabel style={font=\color{black}},
xlabel={$p$},
ylabel style={font=\color{black}},
ylabel={$\| \bm \sigma - \bm \sigma_h \|_{E}$},
axis background/.style={fill=white},
xmajorgrids,
ymajorgrids,
yminorgrids
]
\addplot  [color=blue, line width=2.0pt]
  table[row sep=crcr]{%
1	0.0644449217769595\\
2	0.00795915487418407\\
3	0.000480495273951775\\
4	1.79535882589067e-05\\
5	6.69507399238289e-07\\
6	9.78511106583065e-08\\
};
\end{axis}

\end{tikzpicture}%
\caption{Test case of Section~\ref{sec:unsteady}.}\label{fig:testcase2_convergence_p}
\end{subfigure}
\hspace{0.3cm}
\begin{subfigure}[b]{0.45\textwidth}
\centering
\centering
\includegraphics[width = 1.0\textwidth]{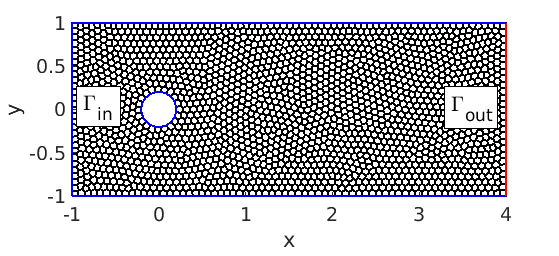}
\caption{Test case of Section~\ref{sec-sub:fac}.}\label{fig:domain_flow}
\end{subfigure}
\caption{Left: Test case of Section~\ref{sec:unsteady}. Semi-logy plot of the computed error $\| \bm \sigma -\bm \sigma_h \|_{E}$ as a function of the polynomial degree $p = 1,...,6$, by fixing $T=0.25$ and \textcolor{black}{$\Delta t = 1e-3$} and $100$ mesh elements. Right: Test case of Section~\ref{sec-sub:fac}. Polygonal mesh with $2000$ elements of the rectangular domain $\Omega$ with a circular hole. Dirichlet boundary is highlighted in blue (up, left, bottom), while the Neumann boundary is in red (right).}
\end{figure}
\noindent $\boldsymbol{F}(\bm x,t)$ and the boundary data are computed accordingly.
In Figure~\ref{fig:testcase2_convergence} we report the computed error $\| \bm \sigma -\bm \sigma_h \|_{E}$ as a function of the discretization parameters. In particular, on the left, we show the log-log plot of the error as a function of the mesh size $h$ for different values of the polynomial degree $p=1,...,4$, by setting $\theta = \frac12$ (Crank-Nicolson), $T = 0.25$ and \textcolor{black}{$\Delta t = 1e-3$}. While on the right, the same quantity is computed versus the time step $\Delta t$ with $\theta=1$ (Implicit Euler) and $\theta = \frac12$ (Crank-Nicolson) in \eqref{eq:tetametodo}, by fixing the mesh elements equal to $400$ and $p=4$. It is possible to observe that the results confirm the theoretical bound in Theorem \ref{thm:convergence}. Finally, in Figure~\ref{fig:testcase2_convergence_p} we plot the computed error as a function of the polynomial degree $p$, by fixing $\theta = \frac12$ (Crank-Nicolson), $T = 0.25$, \textcolor{black}{$\Delta t = 1e-3$} and $100$ mesh elements. We can observe that the method attains exponential convergence, even this is not covered by our theoretical analysis.

\subsection{Pressure and velocity data recovery}\label{sec:recovery}
\begin{figure}
    \centering

\begin{subfigure}[b]{0.48\linewidth}
    \centering
\includegraphics[width=1.3\linewidth]{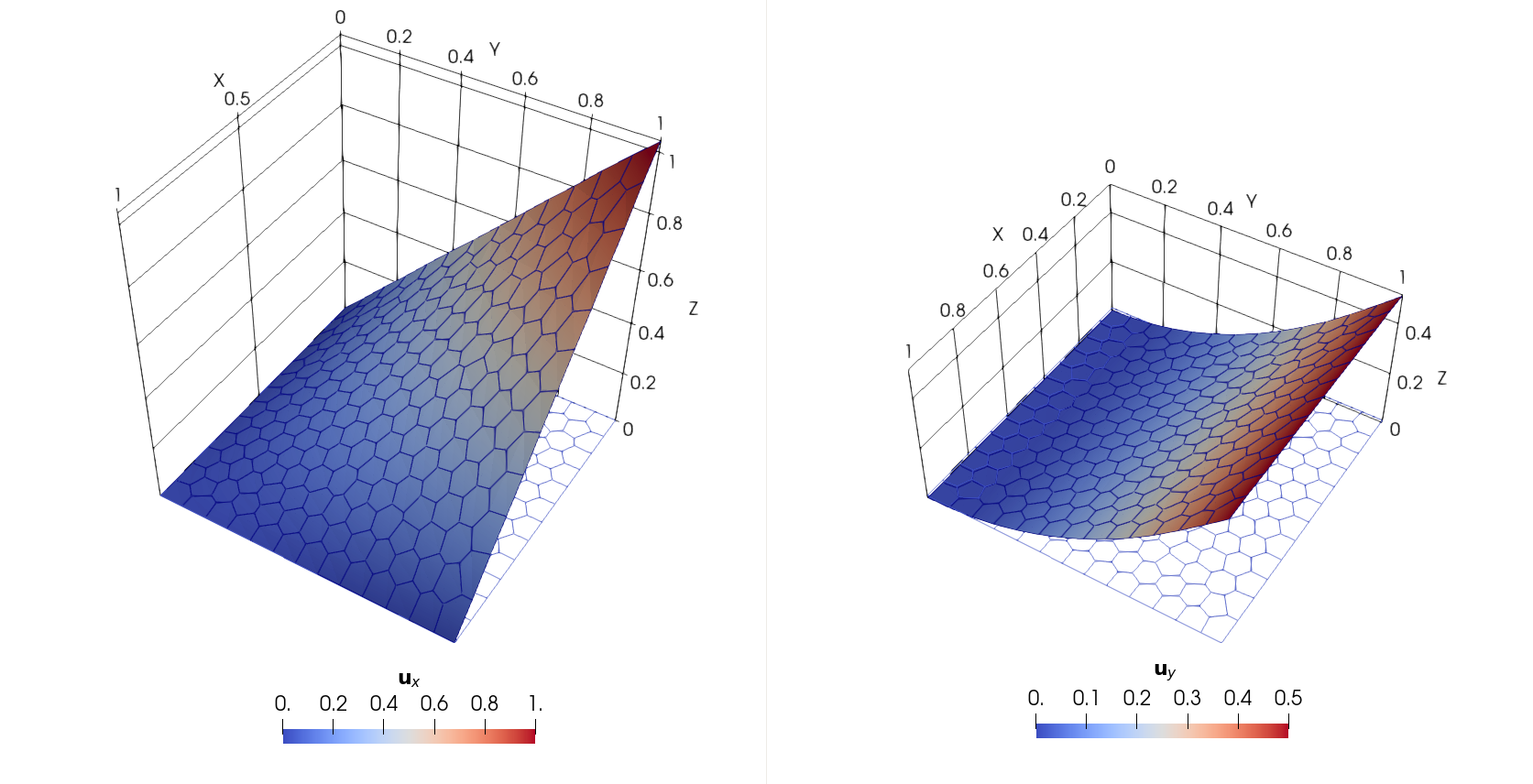}
 \end{subfigure}
 \hspace{0.02\linewidth}
  \begin{subfigure}[b]{0.48\linewidth}
  \centering
\includegraphics[width=0.6\linewidth]{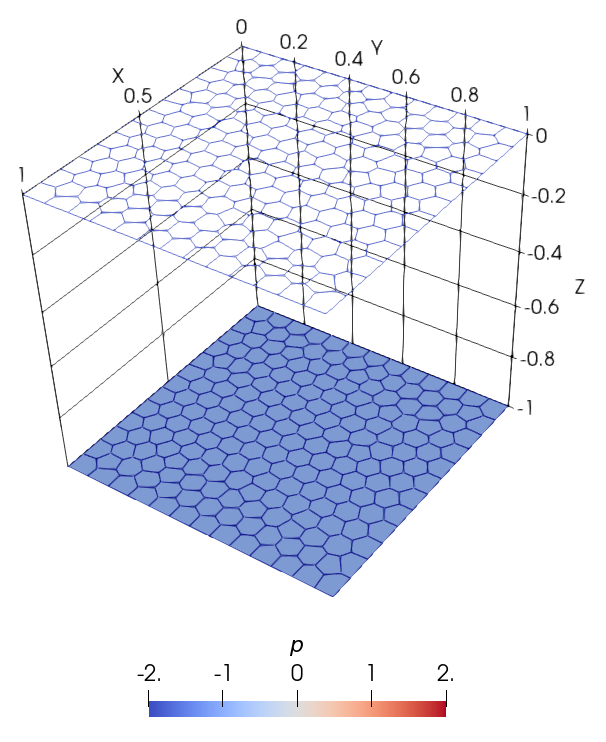}
  \end{subfigure}
    \caption{Test case of Section~\ref{sec:recovery}. Computed velocity field ($u_{x}$ and $u_{y}$) and computed pressure field ($p$) at final time $T = 1$.}
    \label{fig:vel_pre_rec}
\end{figure}
This test case aims to check the capability of the proposed approach to recover the pressure $p$ and the velocity field $\bm u$ of the velocity-pressure Stokes formulation, as explained in Remark~\ref{rem-sub:p_and_u_rec}. We consider $\Omega = (0,1)^2$ and set Dirichlet boundary conditions on all the boundary, except for the right edge where we impose Neumann conditions. 
Then, we consider 
\begin{align*}
\textcolor{black}{\boldsymbol{u}(x,y,t)} = t^2 
\begin{bmatrix}(1-x)y \\ \frac{1}{2}y^2\end{bmatrix},
&& \textcolor{black}{p(x,y,t)} = -\mu t^2,
\end{align*}
as solution of \eqref{eq:stokes_primal}, that lead to the following stress tensor 
$$ \textcolor{black}{\boldsymbol{\sigma}(x,y,t)} = \mu t^2 \begin{bmatrix}1-y & 1-x \\ 0 & 1+y\end{bmatrix},   
$$
as solution of \eqref{eq:stokes_dual}. We consider $\mu = 1$, a mesh with $200$ polygonal elements, the polynomial degree $p=3$, the final time $T=1$, and the time integration step $\Delta t = 1e-2$ with $\theta=\frac{1}{2}$ in \eqref{eq:tetametodo}. To recover the velocity field $\bm u$ a composite trapezoidal quadrature rule has been applied. As one can see from Figure~\ref{fig:vel_pre_rec}, the reconstructed solution at time $T$ perfectly matches with the analytical one.

\subsection{Flow around a cylinder}\label{sec-sub:fac}
In this example, we consider problem \eqref{eq:eqnn} in 
a rectangular domain having a circular hole as depicted in Figure~\ref{fig:domain_flow}. In particular, we set $\Omega = (-1,4)\times(-1,1) \setminus B_{0.2}(0,0)$, being $B_r(x_c,y_c)$ the circle of center $(x_c,y_c)$ and radius $r$, cf. Figure~\ref{fig:domain_flow}. 
We set  $\bm \sigma \bm n = \bm 0$  on $\Gamma_{\text{out}} = \{4\}\times (-1,1)$, $\bm \nabla \cdot \bm \sigma = ((1-y^2),0)^\top$ on $\Gamma_{\text{in}} = \{-1\}\times(-1,1)$, and use  $\bm \nabla \cdot \bm \sigma = \bm 0$ in the remaining part of the boundary. The medium is supposed to be at rest at the initial time, i.e., $\bm \sigma_0 = \bm 0$, and we also  consider $\boldsymbol{f} = \boldsymbol{0}$.
%
%
\begin{figure}
    \centering
    \includegraphics[width = 1\textwidth]{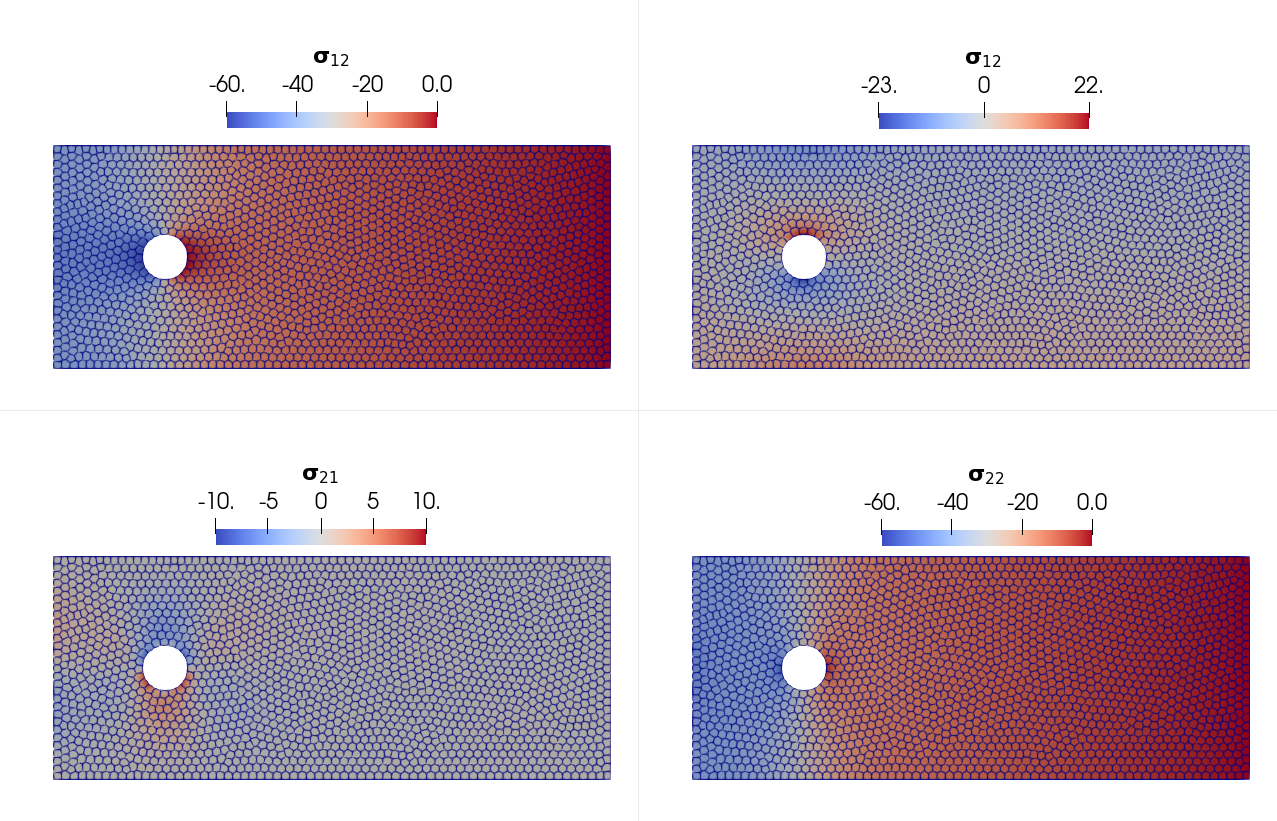}
    \caption{Test case of Section~\ref{sec-sub:fac}. Computed stress tensor field $\bm \sigma_h$ at final time $T=1$.}
    \label{fig:sigma_FAC}
\end{figure}
For the numerical simulation, we set $\mu=2$ in \eqref{eq:eqnn}, employ a mesh made by $2000$ elements, fix the polynomial degree $p_\kappa=3$ for any $\kappa \in \mathcal{T}_h$, choose the final time $T=1$ and the time step $\Delta t = 1e-2$ for the Crank-Nicolson scheme in \eqref{eq:tetametodo}. 
In Figure~\ref{fig:sigma_FAC} we report the computed stress tensor $\bm \sigma_h$ at the final time $T=1$, while in Figure~\ref{fig:pressureFAC}, resp. Figure~\ref{fig:vely_FAC}, we plot the pressure, resp. velocity field, computed according to Remark~\ref{rem-sub:p_and_u_rec}. As in the previous example, to recover the velocity field $\bm u$ a composite trapezoidal quadrature rule has been applied. 
We compare our results to those obtained with the FEniCS \texttt{https://fenicsproject.org/} software, by solving the problem 
with a dG method on triangular meshes made by $7581$ elements and fixing the polynomial degree equal to $2$ for both velocity and pressure variables.
It is possible to see that the proposed PolydG method is able to reproduce correctly the physics of the system, at a much lower computational cost. 
\begin{figure}
    \centering
    \includegraphics[width = 1\textwidth]{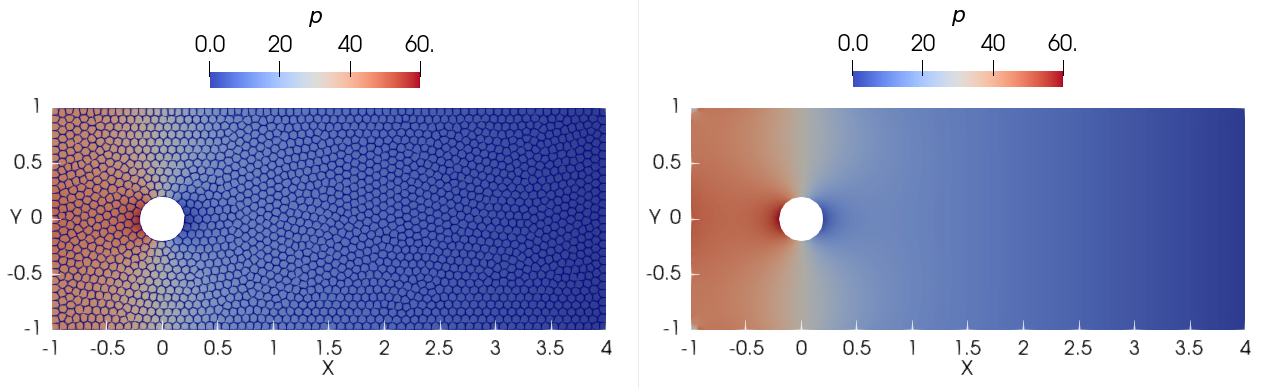}
    \caption{Test case of Section~\ref{sec-sub:fac}. Left: computed pressure field $p_h$ at final time $T=1$ by means of the relation $p_h= -\frac12 tr(\bm \sigma_h)$. Right: computed pressure field $p_h$ obtained with FEniCS.}
    \label{fig:pressureFAC}
\end{figure}
\begin{figure}
    \centering
    \includegraphics[width = 1\textwidth]{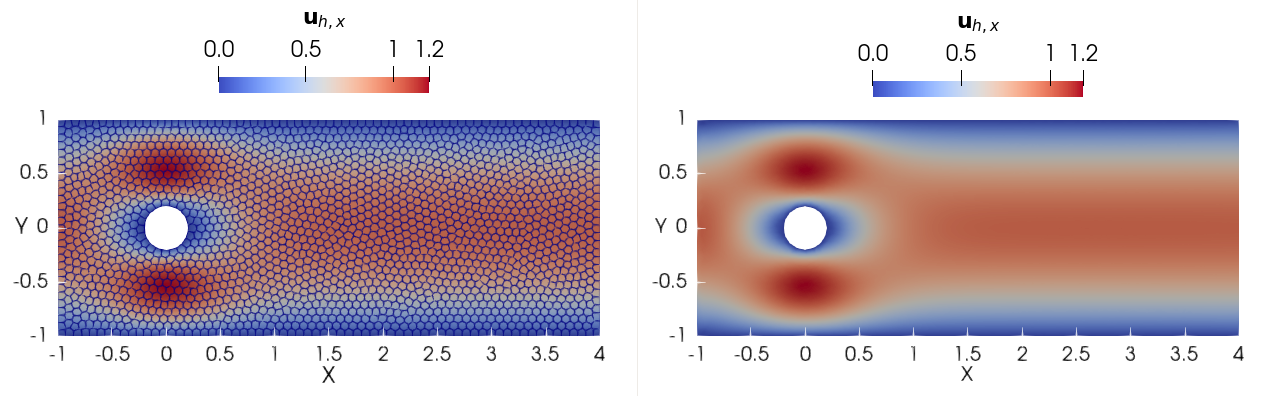}
    \includegraphics[width = 1\textwidth]{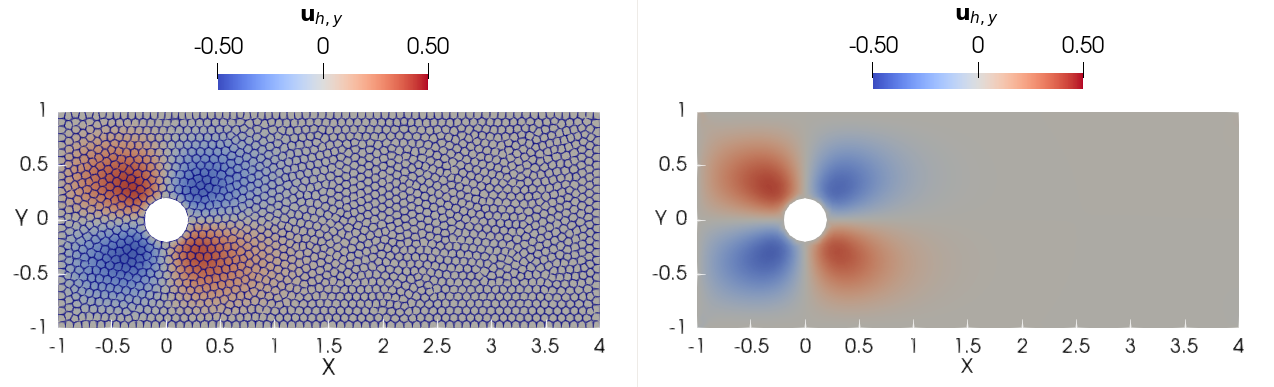}
    \caption{Test case of Section~\ref{sec-sub:fac}. Left: computed velocity field $\bm u_{h}$ at the final time $T=1$ by means of the relation $\bm u_h =\bm u_{0,h} + \int_0^T \boldsymbol{\nabla}\cdot\boldsymbol{\sigma}_h(s)\, ds $. Right: computed velocity field $\bm u_{h}$  obtained with FEniCS.}
    \label{fig:vely_FAC}
\end{figure}
\begin{figure}
    \centering

\begin{subfigure}[b]{0.5\linewidth}
    \centering
\includegraphics[width=0.9\linewidth]{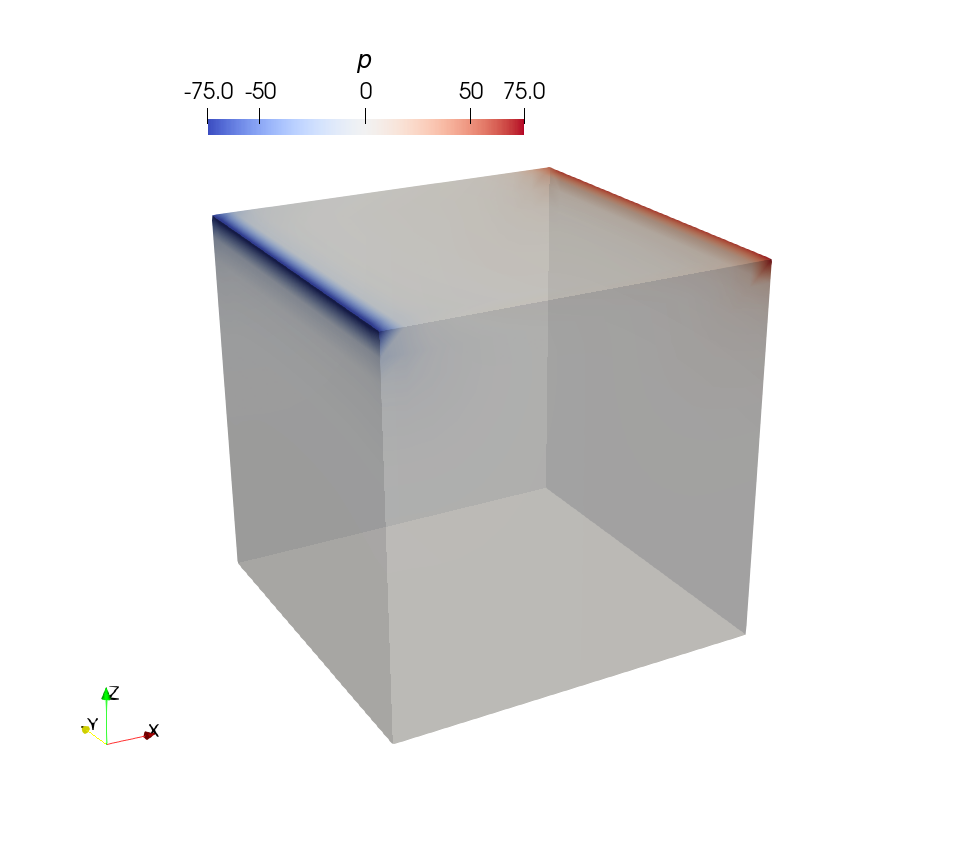}
 \end{subfigure}
 \hspace{0.02\linewidth}
  \begin{subfigure}[b]{0.46\linewidth}
  \centering
\includegraphics[width=1\linewidth]{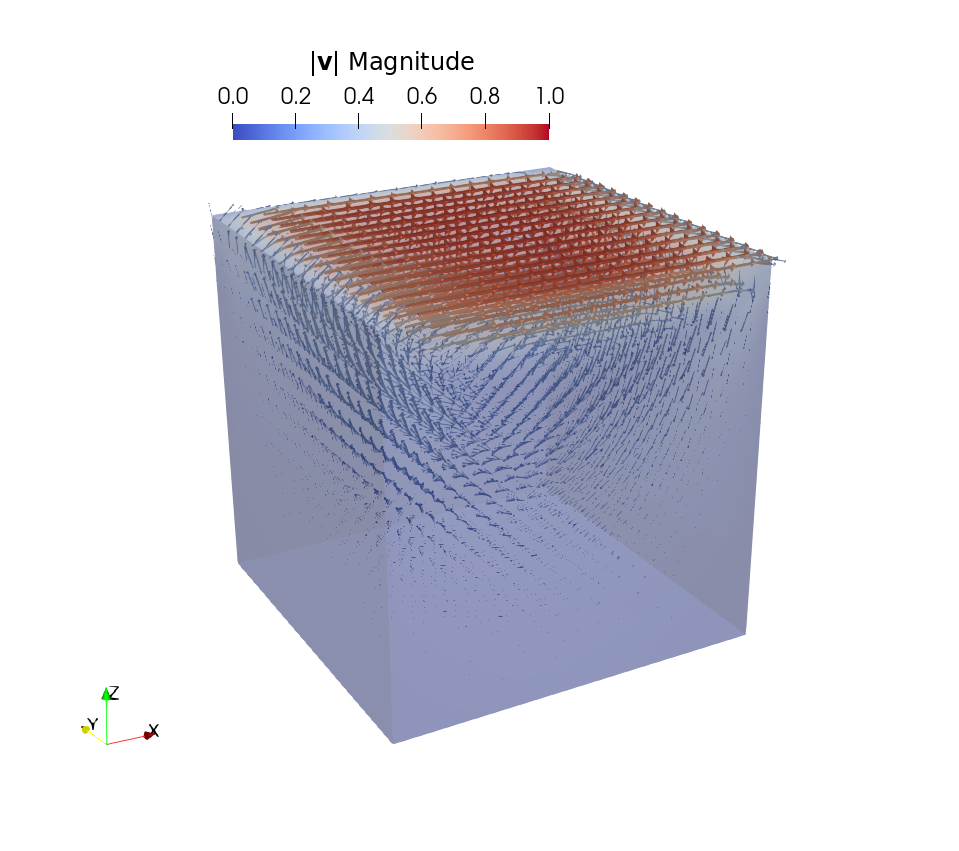}
  \end{subfigure}
    \caption{Test case of Section~\ref{sec-sub:3dtest}, lid-driven cavity. Left: computed pressure field $p_h$ by means of the relation $p_h= -\frac12 tr(\bm \sigma_h)$. Right: computed velocity field $\bm u_{h}$ by means of the relation $\bm u_h =\bm u_{0,h} + \int_0^T \boldsymbol{\nabla}\cdot\boldsymbol{\sigma}_h(s)\, ds $. Final time $T=1$. }
    \label{fig:pressureCavity}
\end{figure}
\begin{figure}
    \centering

\begin{subfigure}[b]{0.5\linewidth}
    \centering
\includegraphics[width=0.9\linewidth]{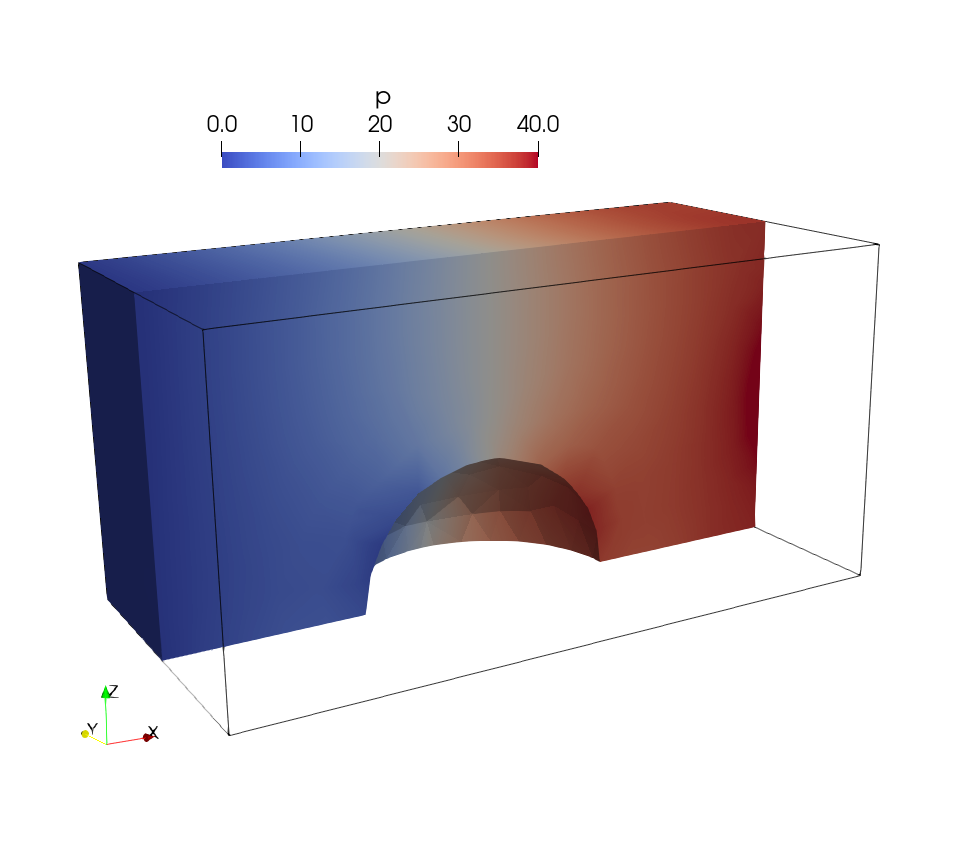}
 \end{subfigure}
 \hspace{0.02\linewidth}
  \begin{subfigure}[b]{0.46\linewidth}
  \centering
\includegraphics[width=1\linewidth]{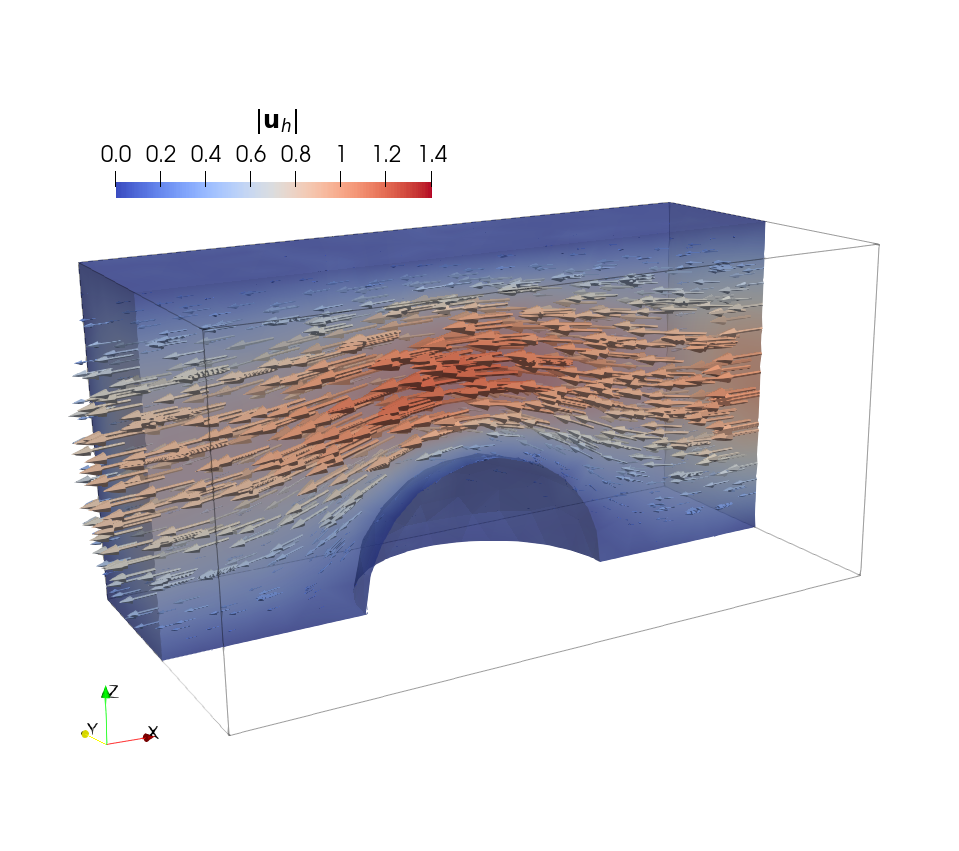}
  \end{subfigure}
    \caption{Test case of Section~\ref{sec-sub:3dtest}, channel-with-obstacle. Left: computed pressure field $p_h$ by means of the relation $p_h= -\frac12 tr(\bm \sigma_h)$. Right: computed velocity field $\bm u_{h}$ by means of the relation $\bm u_h =\bm u_{0,h} + \int_0^T \boldsymbol{\nabla}\cdot\boldsymbol{\sigma}_h(s)\, ds $. Final time $T=1$.}
    \label{fig:obstacle3d}
\end{figure}

\subsection{3D test cases}\label{sec-sub:3dtest}
In this last section, we propose two 3D test cases. We first investigate the classic lid-driven cavity problem, while in the second test, we consider a parallelepiped domain with a spherical cavity on the lower face (a channel-with-obstacle-type problem). In the first test, we set $\Omega = (0,1)^3$, \textcolor{black}{$\Gamma_N = \emptyset$}, and impose $\bm \nabla \cdot \bm \sigma = (1,0,0)^\top$ on $\Gamma_{\text{top}} = (0,1) \times (0,1) \times \{1\}$, which enforces a constant velocity flux along the x-axis. \textcolor{black}{We also set $\bm \nabla \cdot \bm \sigma = \boldsymbol{0}$ in the remaining part of the boundary}, and consider $\boldsymbol{f} = \boldsymbol{0}$  and initial condition $\boldsymbol{\sigma}_0 = \boldsymbol{0}$ in \eqref{eq:eqnn}.
\textcolor{black}{We point out that the previous boundary datum does not belong to the space
$\bm H^{1/2}(\partial \Omega)$. Thus, the regularity of $\bm g_D$ required for the well-posedness of the weak problem \eqref{eq:weak_stress} is not fully satisfied.
Nonetheless, our numerical results will demonstrate that the proposed method performs well also in this benchmark case.}
For the numerical simulation, we set $\mu=0.5$, consider a mesh made by $82944$ tetrahedral elements, impose $\alpha = 40$ in (\ref{def:penalty}), and fix the polynomial degree $p_\kappa=1$ for any $\kappa \in \mathcal{T}_h$. Concerning the time discretization, for the sake of simplicity and to ease the computational cost, we solve only one step of the implicit Euler method ($\theta = 1$ in \eqref{eq:tetametodo}), assuming $\Delta t = 1$. We follow this approach since we aim to validate only the spatial discretization, while the time scheme has already been investigated in previous test cases.  In Figure~\ref{fig:pressureCavity}, we plot the reconstructed pressure and velocity field, computed according to Remark~\ref{rem-sub:p_and_u_rec}. 
Finally, as the last test case, we consider a domain similar to the one presented in Section \ref{sec-sub:fac}, adapted to the 3D case: we set $\Omega = (0,4) \times (0,2) \times (0,2) \setminus B_{0.8}(2,1,0)$. We impose $\bm \nabla \cdot \bm \sigma = (-sin(\frac{\pi}{2}y)sin(\frac{\pi}{2}z),0,0)^\top$ on $\Gamma_{\text{in}} = \{4\} \times (0,2) \times (0,2)$, \textcolor{black}{$\bm \sigma \bm n = \bm 0$  on $\Gamma_{\text{out}} =\{0\} \times (0,2) \times (0,2)$, and $\bm \nabla \cdot \bm \sigma = \boldsymbol{0}$ in the remaining part of the boundary}. We set $\boldsymbol{f} = \boldsymbol{0}$ and $\boldsymbol{\sigma}_0 = \boldsymbol{0}$ as initial condition. For the numerical simulation, we fix $\mu = 2$, consider a mesh composed of $8370$ tetrahedral elements, $\alpha = 55$ in (\ref{def:penalty}), and $p_\kappa=2$ for any $\kappa \in \mathcal{T}_h$. As before, we perform a single step of the implicit Euler method with $\Delta t = 1$, focusing on validating the spatial discretization. We report in Figure~\ref{fig:obstacle3d} the pressure and the velocity field, reconstructed again, as in the previous examples, starting from the numerical pseudo-stress.

\section{Conclusions}\label{sec:conclusions}
\noindent In this work we have presented a theoretical and numerical analysis for a PolydG approximation of the unsteady Stokes problem written in its pseudo-stress formulation. 
We have proved stability results for both the continuous problem as well as for the semi- and fully-discrete schemes. We established an error estimate for the fully-discrete problem, where the PolydG scheme is combined with the $\theta$-method time integration. Numerical tests demonstrated the theoretical estimates, as well as shown that the method can be effectively employed to solve problems of physical interest.
Possible further developments may include the analysis of the coupling between the Stokes and the Biot system, through appropriate interface conditions and the generalization to non-Newtonian fluid flow. \textcolor{black}{Another possible extension could consist in the study of the incompressible nonlinear Navier–Stokes equations. Generalizing the discretization scheme and the analysis could be nontrivial, as the presence of the convective term may require a significant modification of the problem formulation.}


\appendix 

\section{Proofs of the theoretical Lemmas}\label{sec:appendix}

\textbf{Proof of Lemma~\ref{lemma:dev_div}} 
Let $\boldsymbol{\sigma} \in \mathbb{H}_{0,\Gamma_N}({\rm div}, \Omega)$ assuming $|\Gamma_N|>0$. 
\begin{enumerate}
    \item 
We start by establishing \eqref{eq:devdiv}.
There exists $\boldsymbol{v} \in \bm H^1_{0,\Gamma_D}(\Omega)$
such that 
$\boldsymbol{\nabla}\cdot\boldsymbol{v} = \textrm{tr}(\boldsymbol{\sigma})$ and $\left\|\boldsymbol{v}\right\|_{\bm H^{1}(\Omega)} \lesssim  \left\|\textrm{tr}(\boldsymbol{\sigma})\right\|_{L^{2}(\Omega)}$ (cf. \cite{Bertoluzza.ea:17, Botti.Mascotto:24}). Then, it holds
\begin{align*}
\left\|\textrm{tr}(\boldsymbol{\sigma})\right\|_{ L^{2}(\Omega)}^2 & = ( \textrm{tr}(\boldsymbol{\sigma}), \boldsymbol{\nabla}\cdot\boldsymbol{v})_\Omega = 
 (\boldsymbol{\sigma}, \textrm{tr}(\boldsymbol{\nabla}\boldsymbol{v}) \mathbb{I}_d )_\Omega \\ & = 
( \boldsymbol{\sigma} ,  d\ \boldsymbol{\nabla}\boldsymbol{v} - d\ {\rm \textbf{dev}}(\boldsymbol{\nabla}\boldsymbol{v})  )_\Omega \\
& = - d ( {\rm \textbf{dev}}(\boldsymbol{\sigma}) , \boldsymbol{\nabla}\boldsymbol{v})_\Omega - 
d ( 
\boldsymbol{\nabla}\cdot\boldsymbol{\sigma}, \boldsymbol{v} )_\Omega + 
d  (\boldsymbol{\sigma}\boldsymbol{n},\boldsymbol{v} )_{\partial\Omega} \\
& = - d ( {\rm \textbf{dev}}(\boldsymbol{\sigma}) , \boldsymbol{\nabla}\boldsymbol{v})_\Omega - 
d ( 
\boldsymbol{\nabla}\cdot\boldsymbol{\sigma}, \boldsymbol{v} )_\Omega,
\end{align*}
where we first employed the definition of the trace of a tensor, and subsequently the relationship between the trace and the deviatoric operator, along with integration by parts. We also exploited the fact that $\boldsymbol{\sigma}\in \mathbb{H}_{0,\Gamma_N}({\rm div}, \Omega)$ and, as a result, $\langle \boldsymbol{\sigma}\bm{n}, \bm{v}\rangle_{\partial\Omega}=0$ since $\bm{v}$ vanishes on $\Gamma_D$. 
Then, using the Cauchy-Schwarz inequality, we have
\begin{align*} \left\|\textrm{tr}(\boldsymbol{\sigma})\right\|_{ L^{2}(\Omega)}^2 & \lesssim    \left\|{\rm \textbf{dev}}(\boldsymbol{\sigma})\right\|_{\mathbb{L}^{2}(\Omega)}  \left\|\boldsymbol{\nabla}\boldsymbol{v}\right\|_{\mathbb{L}^{2}(\Omega)} + \left\| 
\boldsymbol{\nabla}\cdot\boldsymbol{\sigma}\right\|_{\bm L^{2}(\Omega)} \left\|\boldsymbol{v}\right\|_{\bm L^{2}(\Omega)}  \\
& \lesssim \Big(  \left\|{\rm \textbf{dev}}(\boldsymbol{\sigma})\right\|_{\mathbb{L}^{2}(\Omega)}  + \left\| 
\boldsymbol{\nabla}\cdot\boldsymbol{\sigma}\right\|_{\bm L^{2}(\Omega)} \Big) \left\|\boldsymbol{v}\right\|_{\bm H^{1}(\Omega)} \\
& \lesssim  \Big(  \left\|{\rm \textbf{dev}}(\boldsymbol{\sigma})\right\|_{\mathbb{L}^{2}(\Omega)}  + \left\| 
\boldsymbol{\nabla}\cdot\boldsymbol{\sigma}\right\|_{\bm L^{2}(\Omega)} \Big) \left\|\textrm{tr}(\boldsymbol{\sigma})\right\|_{ L^{2}(\Omega)},
\end{align*}
from which we get 
the conclusion, i.e., 
\begin{align*}
\left\|\boldsymbol{\sigma}\right\|_{\mathbb{L}^{2}(\Omega)} & \le \left\|{\rm \textbf{dev}}(\boldsymbol{\sigma})\right\|_{\mathbb{L}^{2}(\Omega)} + \left\|\textrm{tr}(\boldsymbol{\sigma})\right\|_{L^{2}(\Omega)}  \lesssim   \left\|{\rm \textbf{dev}}(\boldsymbol{\sigma})\right\|_{\mathbb{L}^{2}(\Omega)}  + \left\| 
\boldsymbol{\nabla}\cdot\boldsymbol{\sigma}\right\|_{\bm L^{2}(\Omega)}. 
\end{align*}

\item
We now move to the proof of the trace inequality \eqref{eq:traceHdiv}. First, we recall that the trace operator $\gamma: \bm H^1(\Omega) \rightarrow \bm H^{\frac12}(\partial\Omega)$ is onto and has a continuous lifting. Therefore, by definition of the dual norm, we have
$$
\textcolor{black}{\left\|\boldsymbol{\sigma}\boldsymbol{n}\right\|_{\bm H^{-1/2}(\partial\Omega)} }
=
\sup_{\bm v   \in  \boldsymbol{H}^{1/2}(\partial\Omega)} 
\frac{\langle \boldsymbol{\sigma}\boldsymbol{n}, \bm v  \rangle_{\partial\Omega}}{\left\|\bm v\right\|_{\frac12,\partial\Omega}}
\lesssim \sup_{\bm w \in \boldsymbol{H}^1(\Omega)} 
\frac{\langle \boldsymbol{\sigma}\boldsymbol{n}, \gamma(\bm w)  \rangle_{\partial\Omega}}{\left\|\bm w\right\|_{1,\Omega}}.
$$
Applying the Green formula of \cite[Lemma 2.1.1]{Boffi2013} followed by the Cauchy--Schwarz inequality we obtain 
$$
\begin{aligned}  
\textcolor{black}{\left\|\boldsymbol{\sigma}\boldsymbol{n}\right\|_{\bm H^{-1/2}(\partial\Omega)}}
&\lesssim
\sup_{\bm w \in \boldsymbol{H}^1(\Omega)} 
\frac{\int_\Omega \bm\nabla\cdot\bm\sigma \cdot \bm w + 
\int_\Omega \bm\sigma : \bm\nabla\bm w
}{\left\|\bm w\right\|_{1,\Omega}} \lesssim 
\left\|\boldsymbol{\sigma}\right\|_{\mathbb{L}^{2}(\Omega)}  + \left\| 
\boldsymbol{\nabla}\cdot\boldsymbol{\sigma}\right\|_{\mathbb{L}^{2}(\Omega)}.
\end{aligned}
$$
The thesis follows by bounding the $L^2$-norm of $\boldsymbol{\sigma}$ with the ${\rm dev}$-${\rm div}$ inequality.
\end{enumerate}

\begin{lemma}\label{lemma:cont_coerc_A} 
Coercivity and continuity of $\mathcal{A}(\cdot, \cdot)$.
\begin{align}
    \mathcal{A}(\boldsymbol{\sigma}_h, \boldsymbol{\sigma}_h) & \gtrsim |\boldsymbol{\sigma}_h|_{dG}^2  & \forall \boldsymbol{\sigma}_h \in \boldsymbol{V}_h, \label{eq:coercivity_dg}\\
    \mathcal{A}(\boldsymbol{\sigma}, \boldsymbol{\tau}_h) & \lesssim \trinorm{\boldsymbol{\sigma}}_{dG} |\boldsymbol{\tau}_h|_{dG}  & 
    \forall \boldsymbol{\sigma} \in \mathbb{H}^2(\mathcal{T}_h),  \forall \boldsymbol{\tau}_h \in \boldsymbol{V}_h.     \label{eq:continuity_dg}
\end{align} \\
The coercivity bound holds provided the stability parameter $\alpha$ appearing
in \eqref{def:penalty} is chosen sufficiently large.

\end{lemma}
\begin{proof}
 See \cite[Lemma A.3]{Antonietti2021}.   
\end{proof}

\begin{lemma}\label{lemma:sigma_l2}
Let $\boldsymbol{\sigma}_h \in \bm V_h$. Then, it holds
\begin{equation*}
 \left\|\boldsymbol{\sigma}_h\right\|^2_{\mathbb{L}^2(\Omega)} \lesssim \left\|\mu^{-\frac12}\boldsymbol{{\rm dev}} (\boldsymbol{\sigma}_h)\right\|^2_{\mathbb{L}^2(\Omega)} +|\boldsymbol{\sigma}_h|_{dG}^2.
    \end{equation*}
\end{lemma} 

\begin{proof}
Similarly to the proof of Lemma~\ref{lemma:dev_div}, we start by observing that there is $\boldsymbol{v} \in \bm H^1_{0,\Gamma_D}(\Omega)$ such that 
$\boldsymbol{\nabla}\cdot\boldsymbol{v} = \rm{tr}(\boldsymbol{\sigma}_h)$ and $\left\|\boldsymbol{v}\right\|_{\bm H^{1}(\Omega)} \lesssim \left\|\rm{tr}(\boldsymbol{\sigma}_h)\right\|_{L^{2}(\Omega)}.$
Then, we infer
\begin{align}\label{eq:trace_l2}
\left\|{\rm{tr}}(\boldsymbol{\sigma}_h)\right\|_{ L^2(\Omega)}^2 
=  d ( \boldsymbol{\sigma}_h, \boldsymbol{\nabla}\boldsymbol{v} )_{\mathcal{T}_h} - 
d(  \boldsymbol{{\rm dev}}(\boldsymbol{\sigma}_h), \boldsymbol{\nabla}\boldsymbol{v})_{\mathcal{T}_h} .
\end{align}
We focus on the first term on the right-hand side, we integrate it by parts, recall that $\boldsymbol{v} \in \bm H^1_{0,\Gamma_D}(\Omega)$ implies $[[\boldsymbol{v}]]=\bm0$ on $\mathcal{F}_h^{I,D}$, and apply the Cauchy-Schwarz inequality to get
\begin{align*}
     ( \boldsymbol{\sigma}_h, \boldsymbol{\nabla}\boldsymbol{v} )_{\mathcal{T}_h} 
     & = - ( \nabla \cdot \boldsymbol{\sigma}_h, \boldsymbol{v})_{\mathcal{T}_h} 
    +   \langle [[\boldsymbol{\sigma}_h\boldsymbol{n}]], \boldsymbol{v} \rangle_{\mathcal{F}_h^{I,N}} \\
    & \lesssim \left\| \nabla_h \cdot \boldsymbol{\sigma}_h \right\|_{\bm L^2(\Omega)} \left\| \boldsymbol{v} \right\|_{\bm L^2(\Omega)} 
+  \left\| \ \gamma^{\frac12}[[\boldsymbol{\sigma}_h\boldsymbol{n}]] \ \right\|_{\mathcal{F}_h^{I,N}} \left\| \ \gamma^{-\frac12}\boldsymbol{v} \ \right\|_{\mathcal{F}_h^{I,N}}. 
\end{align*}
Recalling the continuous local trace inequality (see \cite[Lemma 1.49.]{DiPietro2012}), which gives
\begin{align*}
    \sum_{F\in \mathcal{F}_h^{I,N}} \left\| \ \gamma^{-\frac12}\boldsymbol{v} \ \right\|_{\bm{L}^2(F)}^2
    \lesssim
    \sum_{\kappa\in \mathcal{T}_h} \left\| \ \gamma^{-\frac12}\boldsymbol{v} \ \right\|_{\bm{L}^2(\partial\kappa)}^2
    \lesssim 
    \sum_{\kappa\in \mathcal{T}_h} \Big(
    \left\| \boldsymbol{v} \right\|_{\bm L^2(\kappa)}^2 + h_\kappa^2 \ \left\| \boldsymbol{\nabla v} \right\|_{\mathbb{L}^2(\kappa)}^2 \Big),
\end{align*}
it is inferred that
\begin{align*}
     (\boldsymbol{\sigma}_h ,\boldsymbol{\nabla}\boldsymbol{v})_{\mathcal{T}_h} \lesssim
     \Big(  \left\| \nabla_h \cdot \boldsymbol{\sigma}_h \right\|_{\bm L^2(\Omega)} + \left\| \ \gamma^{\frac12}[[\boldsymbol{\sigma}_h\boldsymbol{n}]] \ \right\|_{\mathcal{F}_h^{I,N}} \Big) \left\| \boldsymbol{v} \right\|_{\bm H^1(\Omega)}.
\end{align*}
Plugging the above estimate into \eqref{eq:trace_l2}, we obtain
\begin{align*}
\left\|\rm{tr}(\boldsymbol{\sigma}_h)\right\|_{L^2(\Omega)}
\lesssim  \left\|\mu^{-\frac12}\boldsymbol{{\rm dev}}(\boldsymbol{\sigma}_h)\right\|_{\mathbb{L}^{2}(\Omega)}  + \left\| \nabla_h \cdot \boldsymbol{\sigma}_h \right\|_{\bm L^2(\Omega)}
+ \left\| \ \gamma^{\frac12}[[\boldsymbol{\sigma}_h\boldsymbol{n}]] \ \right\|_{\mathcal{F}_h^{I,N}}
\end{align*}
and, as a result of $\boldsymbol{\sigma}_h = \rm{ \textbf{dev}}(\boldsymbol{\sigma}_h) + d^{-1}\rm{tr}(\boldsymbol{\sigma}_h)\mathbb{I}_d$ and the definition of the $|\cdot|_{dG}$ seminorm, the thesis follows.
\end{proof}

\section{Aknowledgements}
This work is funded
by the European Union (ERC SyG, NEMESIS, project number 101115663). Views and opinions expressed are however
those of the authors only and do not necessarily reflect those of the European Union or the European Research Council Executive Agency.
P.F.A. and I.M. have been partially funded by ICSC—Centro Nazionale di Ricerca in High Performance Computing, Big Data, and Quantum Computing funded by European Union—NextGenerationEU.
M.B., I.M., and P.F.A. are members of INdAM-GNCS. 
The work of M.B. has been partially supported by the INdAM-GNCS project CUP E53C23001670001.
The work of I.M. has been partially supported by the INdAM-GNCS project CUP E53C22001930001.
The present research is part of the activities of "Dipartimento di Eccellenza 2023-2027".


\end{document}